\RequirePackage{fix-cm}
\documentclass[smallextended]{svjour3}       
\smartqed  
\usepackage{graphicx}
\usepackage{framed,multirow}
\usepackage{latexsym}
\usepackage{amsmath,amsfonts,amssymb}%
\usepackage{enumerate}

\numberwithin{theorem}{section}

\numberwithin{lem}{section}
\newtheorem{defn}{Definition}[section]
\numberwithin{defn}{section}
\numberwithin{remark}{section}
\usepackage{caption}
\usepackage{subcaption}%
\usepackage{color}
\usepackage{hyperref}
\usepackage{lineno}
\numberwithin{equation}{section}
\usepackage{url}

\begin{document}

\title{Flux correction for nonconservative convection-diffusion equation
}


\author{Sergii Kivva       
}


\institute{Sergii Kivva \at
              Institute of Mathematical Machines and System Problems, National Academy of Sciences, Kyiv, Ukraine \\
              \email{skivva@gmail.com}           
}

\date{Received: date / Accepted: date}

\maketitle

\begin{abstract}
Our goal is to develop a flux limiter of the Flux-Corrected Transport method for a nonconservative convection-diffusion equation. For this, we consider a hybrid difference scheme that is a linear combination of a monotone scheme and a scheme of high-order accuracy. The flux limiter is computed as an approximate solution of a corresponding optimization problem with a linear objective function. The constraints for this optimization problem are derived from inequalities that are valid for the monotone scheme and apply to the hybrid scheme. Our numerical results with the flux limiters, which are exact and approximate solutions to the optimization problem, are in good agreement.
\end{abstract}

\section{Introduction}
\label{intro}
The objective of this paper is to develop a flux limiter for the flux-corrected transport (FCT) method for a nonconservative convection-diffusion equation. The numerical solution of such equations arises in a variety of applications such as hydrodynamics, heat, and mass transfer. To the best of our knowledge, we are not aware of any formulas for computing the FCT flux limiter for a nonconservative convection-diffusion equation.

On an interval $[a, b]$, we consider the initial boundary value problem (IBVP) for a nonconservative convection-diffusion equation
\begin{equation}
\label{eq:1_1}
\frac{{\partial \rho }}{{\partial t}} + u(x,t)\frac{{\partial \rho }}{{\partial x}} + \lambda (x,t)\rho  = \frac{\partial }{{\partial x}}\left( {D(x,t)\frac{{\partial \rho }}{{\partial x}}} \right) + f(x,t), \quad t > 0  
\end{equation}
with initial condition
\begin{equation}
\label{eq:1_2}
\rho (x,0) = {\rho ^0}(x)
\end{equation}
where $0 \le D(x,t) \le \mu = const$.

For simplicity and without loss of generality, we assume that the Dirichlet boundary conditions are specified at the ends of the interval $[a,b]$ 
\begin{equation}
\label{eq:1_3}
\rho (a,t) = {\rho _a}(t) 
\end{equation}
\begin{equation}
\label{eq:1_4}
\rho (b,t) = {\rho _b}(t) 
\end{equation}

The two-step FCT algorithm was firstly developed by Boris and Book~\cite{b1} for solving a transient continuity equation. Within this approach, the flux at the cell interface is computed as a convex combination of fluxes of a monotone low-order scheme and a high-order scheme. These two fluxes are combined by adding to one of them (basic flux) a limited flux that is the limited difference between the high-order and low-order fluxes at the cell interface. 
In the classical FCT method, the low-order flux is basic and the additional limited flux is antidiffusive.
Kuzmin and his coworkers~\cite{b7,b6} consider the high-order flux as the basic with an additional dissipative flux. Such approach is now known as algebraic flux correction (AFC).
The procedure of two-step flux correction consists of computing the time advanced low order solution in the first step and correcting the solution in the second step to produce accurate and monotone results. The basic idea is to switch between high-order scheme and positivity preserving low-order scheme to provide oscillation free good resolution in steep gradient areas, while at the same time preserve at least second-order accuracy in smooth regions. Later Zalesak~\cite{b2,b3} extended FCT to multidimensional explicit difference schemes. Since the 1970s, FCT has been widely used in the modeling of various physical processes. Many variations and generalizations of FCT and their applications are given in~\cite{b36}.

In this paper, we derive the flux correction formulas for the nonconservative convection-diffusion equation using the approach proposed in \cite{Kivva}. As in the classical FCT method, we use a hybrid difference scheme consisting of a convex combination of low-order monotone and high-order schemes. According to \cite{Kivva}, finding the flux limiters we consider as a corresponding optimization problem with a linear objective function. The constraints for the optimization problem derive from the inequalities which are valid for the monotone scheme and apply to the hybrid scheme. The flux limiters are obtained as an approximate solution to the optimization problem. 
Numerical results show that these flux limiters produce numerical solutions that are in good agreement with the numerical solutions, the flux limiters of which are calculated from optimization problem and correspond to maximal antidiffusive fluxes.

The advantage of such approach is that the two-step classical FCT method is reduced to one-step. 
For flux corrections in the classical FCT method, it is necessary to know the low-order numerical solution at the current time step. In the proposed approach \cite{Kivva}, it is sufficient to know only the numerical solution at the previous time step.

The paper is organized as follows. In Section \ref{Sec1}, we discretize the IVBP~\eqref{eq:1_1}-\eqref{eq:1_4} by a hybrid scheme. An analog of the discrete local maximum principle for the monotone scheme is given in Section \ref{Sec2}. The optimization problem for finding flux limiters and the algorithm of its solving are described in Section \ref{Sec3}. An approximate solution of the optimization problem is derived in Section \ref{Sec4}. The results of numerical experiments are presented in Section \ref{Sec5}. Concluding remarks are drawn in Section \ref{Sec6}.

\section{Hybrid difference scheme}
\label{Sec1}

In this section, we discretize the IBVP \eqref{eq:1_1}-\eqref{eq:1_4} using a hybrid difference scheme, which is a linear combination of a monotone scheme and a high-order scheme.

On the interval $[a,b]$, we introduce a nonuniform grid $\Omega_h$
\begin{equation}
\label{eq:2_1}
\Omega_h = \left\{ {x_i: \;\; x_i = x_{i-1} + \Delta _{i-1/2}x,\;i = \overline {1,N} ;\; x_0 = a, \, x_{N+1} = b} \right\} 
\end{equation}

Assuming that $u(x,t)$ and $\rho (x,t)$ are sufficiently smooth, we consider some approximations of the convective term in \eqref{eq:1_1}. For this, we integrate it on an interval $[x_{i-1/2},x_{i+1/2}]$ and applying the rectangular approximation method at the point $x_i$, as well as backward and forward differencing for the first-order derivative,  we obtain the following upwind discretization
\begin{equation}
\label{eq:2_2}
\begin{split} 
\int\limits_{x_{i-1/2}}^{x_{i+1/2}} {u \,\frac{\partial \rho }{\partial x}} dx = \Delta {x_i} \left[ {{{\left( {u^+ \, \frac{ \partial \rho }{\partial x}} \right)}_i} + {{\left( {u^- \, \frac{\partial \rho }{\partial x}} \right)}_i}} \right] \\
  = \Delta {x_i} \left[ {u_i^+ \, \frac{(\rho_i - \rho_{i-1})}
  {{\Delta _{i-1/2}x}} + u_i^- \frac{{(\rho_{i+1} - \rho_i)}}
  {{\Delta _{i +1/2}x}}} \right] + O\left( {\Delta x_i^2} \right)  
\end{split}   
\end{equation}
where $\rho_i = \rho(x_i,t)$; $\Delta x_i = (x_{i+1} -x_{i-1})/2$ is the spatial size of the $i$th cell; $u^{\pm}=(u \pm \vert u \vert)/2$.

Applying the left and right rectangular rules for numerical integration and central differencing for the first-order derivative, we have another form of upwind discretization
\begin{equation}
\label{eq:2_3}
\begin{split} 
& \int\limits_{x_{i-1/2}}^{x_{i+1/2}} {u\frac{\partial \rho} {\partial x}} dx = \int\limits_{x_{i-1/2}}^{x_{i+1/2}} {u^+ \, \frac{\partial \rho}{\partial x}} dx + \int\limits_{x_{i-1/2}}^{x_{i+1/2}} {u^- \, \frac{\partial \rho}{\partial x}} dx \\
= & \Delta {x_i}\left[ {u_{i-1/2}^+ \frac{{(\rho_i - \rho_{i-1})}}{{\Delta _{i-1/2}}x} + u_{i+1/2}^- \frac{{(\rho _{i+1} - \rho_i)}}{{\Delta _{i+1/2}}x}} \right] + O\left( {\Delta x_i^2} \right)
\end{split}   
\end{equation}

To obtain an approximation of a higher order, in the rectangular approximation rule at a point $x_i$, we use central differencing for the first-order derivative
\begin{equation}
\label{eq:2_4}
\begin{split} 
& \int\limits_{x_{i-1/2}}^{x_{i+1/2}} {u\frac{{\partial \rho }}{\partial x}} dx = \frac{\Delta x_i}{2}{u_i} \frac{(\rho _{i+1} - \rho _{i-1})}{\Delta x_i} \\ 
+ & M \Delta {x_i} \left( {\Delta _{i+1/2}x - {\Delta _{i-1/2}}x} \right) + O\left( {\Delta x_i^3} \right)
\end{split}   
\end{equation}
where $M=const$.

Applying the trapezoidal rule for numerical integration and central differencing for the first-order derivative, we obtain
\begin{equation}
\label{eq:2_5}
\begin{split} 
\int\limits_{x_{i-1/2}}^{x_{i+1/2}} {u \frac{\partial \rho }{\partial x}} dx = \frac{\Delta {x_i}}{2} \left[ {{u_{i-1/2}} \frac{{(\rho _i - \rho _{i-1})}}{{\Delta _{i-1/2}}x} + {u_{i+1/2}} \frac{(\rho _{i+1} - \rho _i)}{{\Delta _{i+1/2}}x}} \right] + O\left( {\Delta x_i^3} \right)
\end{split}   
\end{equation}

Besides, we rewrite the convective term in \eqref{eq:1_1} as follows:
\begin{equation}
\label{eq:2_5_1}
 u\frac{\partial \rho }{\partial x} = \frac{\partial }{\partial x}\left( {u\rho } \right) - \rho \frac{\partial u}{\partial x}  
\end{equation}

We discretize the terms on the right-hand side of \eqref{eq:2_5_1} by the following difference relations
\begin{equation}
\label{eq:2_5_2}
\int\limits_{x_{i-1/2}}^{x_{i+1/2}} {\frac{\partial (u\rho )}{\partial x}} dx = u_{i+1/2}^+ {\rho _i} + u_{i+1/2}^- {\rho _{i+1}} - u_{i-1/2}^+ {\rho _{i-1}} - u_{i-1/2}^ - {\rho _i} + O\left( {\Delta {x_i}} \right)
\end{equation}
\begin{equation}
\label{eq:2_5_3}
\int\limits_{x_{i-1/2}}^{x_{i+1/2}} {\frac{\partial (u\rho )}{\partial x}} dx = \frac{1}{2}{u_{i+1/2}}(\rho _i + \rho _{i+1}) - u_{i-1/2}(\rho _{i-1} + {\rho _i}) + O\left( {\Delta x_i^2} \right)
\end{equation}
\begin{equation}
\label{eq:2_5_4}
\int\limits_{x_{i-1/2}}^{x_{i+1/2}} {\rho \frac{\partial u}{\partial x}} dx = {\rho _i}\left( {u_{i+1/2} - u_{i-1/2}} \right) + O\left( {\Delta x_i^2} \right)
\end{equation}
Using a convex combination of \eqref{eq:2_5_2} and \eqref{eq:2_5_3} to approximate the divergent term in \eqref{eq:2_5_1}, we discretize the convective term as
\begin{equation}
\label{eq:2_5_4}
\begin{split}  
& \int\limits_{x_{i-1/2}}^{x_{i+1/2}} {u\frac{\partial \rho }{\partial x}} dx = \left[ u_{i+1/2}^+ \rho _i + u_{i+1/2}^- \rho _{i+1} + {\beta _{i+1/2}}\frac{\left| {u_{i+1/2}} \right|}{2}\left( {\rho _{i+1} - \rho _i} \right) \right. \\
& - \left. u_{i-1/2}^+ \rho _{i-1} - u_{i-1/2}^- {\rho _i} - {\beta _{i-1/2}}\frac{{\left| {u_{i-1/2}} \right|}}{2}\left( {{\rho _i} - \rho _{i-1}} \right) \right]  
- {\rho _i}\left( {u_{i+1/2}} - {u_{i-1/2}} \right)\end{split} 
\end{equation}
where $\beta _{i+1/2}$ is the flux limiter for the divergent part in square brackets of the convective flux. For a flux correction of the convective term in the divergent form, we refer to \cite{Kivva}.

Below, to approximate the convective term in \eqref{eq:1_1}, we apply a convex combination of \eqref{eq:2_3} and \eqref{eq:2_5}. Note that
\begin{equation}
\label{eq:2_6}
\begin{split} 
& \frac{1}{2} \left[ {{u_{i+1/2}} \frac{{(\rho _{i+1} - \rho _i)}} {{\Delta _{i+1/2}}x} + u_{i-1/2} \frac{{(\rho _i - \rho _{i - 1})}} {{\Delta _{i-1/2}}x}} \right] 
= \left[ u_{i+1/2}^- \frac{{(\rho _{i+1} - \rho _i)}} {{\Delta _{i+1/2}}x} \right. \\
+ & \left. u_{i-1/2}^+ \frac{(\rho _i - \rho _{i-1})}{{\Delta _{i-1/2}}x} \right] + \left[ {\frac{{\left| {u_{i+1/2}} \right|}}{2} \frac{(\rho _{i+1} - \rho _i)} {{\Delta _{i+1/2}}x} - \frac{\left| {u_{i-1/2}} \right|}{2}\frac{(\rho _i - \rho _{i-1})} {{\Delta _{i-1/2}}x}} \right]	
\end{split}   
\end{equation}
The second term in square brackets on the right-hand side of \eqref{eq:2_6} can be considered as an anti-diffusion.

We approximate \eqref{eq:1_1}-\eqref{eq:1_4} by the following weighted difference scheme
\begin{equation}
\label{eq:2_7}
\frac{y_i^{n+1} - y_i^n}{\Delta t} + h_{i+1/2}^{- ,(\sigma )} + h_{i-1/2}^{+,(\sigma )} + (\lambda y)_i^{(\sigma )} = f_i^{(\sigma )}
\end{equation}
where $y_i^n = y(x_i,t^n)$ is the grid function on $\Omega_h$; $\Delta t$ is the time step; $f_i^{(\sigma)} = \sigma f_i^{n+1} +(1-\sigma) f_i^n, \sigma \in [0,1]$. The numerical flux $h_{i+1/2}^{\pm,n}$ is written in the form
\begin{equation}
\label{eq:2_8}
h_{i \mp 1/2}^ {\pm,n}  = \left( {u_{i \mp 1/2}^ {\pm,n}  + d_i^ {\pm,n}  - \alpha _i^ {\pm,n} r_i^ {\pm, n} } \right)\frac{{\Delta _{i \mp 1/2} y^n}}{\Delta _{i \mp 1/2}x}
\end{equation}
where $\alpha_i^{\pm,n} \in [0,1]$ is the flux limiter; $\Delta_{i+1/2} y^n = y_{i+1}^n -y_i^n$; the coefficients $d_i^{\pm,n}$ and $s_i^{\pm,n}$ are computed as
\begin{equation}
\label{eq:2_9}
d_i^ {\pm,n}  =  \pm \max \left( {0,\frac{D_{i \mp 1/2}^n}{\Delta x_i} - \frac{\left| {u_{i \mp 1/2}^n} \right|}{2}} \right) 
\end{equation}
\begin{equation}
\label{eq:2_10}
r_i^ {\pm,n}  =  \mp \min \left( {0,\frac{D_{i \mp 1/2}^n} {\Delta x_i} - \frac{\left| {u_{i \mp 1/2}^n} \right|}{2}} \right)
\end{equation}
Note that for $\sigma=0$ scheme \eqref{eq:2_7} is explicit and implicit for $\sigma>0$. Let us denote by $y_0^n$ and $y_{N+1}^n$  the values of $\rho(x,t)$ at the left and right ends of the interval $[a,b]$ at time $t^n$.

We rewrite the difference scheme \eqref{eq:2_7} in matrix form as
\begin{equation}
\label{eq:2_11} 
\begin{split} 
& \left[ E +  \Delta t\sigma \left( A^{n+1} + \Lambda ^{n+1} \right) \right]{\boldsymbol y^{n+1}} -\Delta t \left[ (B^-{\boldsymbol \alpha^-} )^{(\sigma)} + (B^+{\boldsymbol \alpha^+})^{(\sigma)} \right] \\
& \qquad \qquad = \left[ {E - \Delta t(1 - \sigma ) \left( A^n + {\Lambda ^n} \right) } \right]{\boldsymbol y^n} + \Delta t{\boldsymbol g^{(\sigma )}}
\end{split}
\end{equation} 
where $(B^\pm {\boldsymbol \alpha^\pm)^{(\sigma)}}=\sigma B^{\pm,n+1} \boldsymbol \alpha^{\pm,n+1} +(1-\sigma) B^{\pm,n} \boldsymbol \alpha^{\pm,n}$; $B^{\pm}=diag \lbrace b_i^{\pm}(\boldsymbol y) \rbrace _{i=1}^N$ is the diagonal matrix;
$E$ is the identity matrix of order $N$; $A=\lbrace a_{ij} \rbrace _i^j$ is tridiagonal square matrices of order $N$; $\Lambda=diag(\lambda_1,\ldots,\lambda_N)$ is the diagonal matrix; $\boldsymbol \alpha ^\pm = (\alpha_1^\pm,\ldots,\alpha_N^\pm)^T \in R^N$ are the numerical vectors of flux limiters; $\boldsymbol g=(g_1,\ldots,g_N)^T$ is the vector of boundary conditions and values of the function $f$ at the points $x_i$. Components of the vector $\boldsymbol g$ are given by 
\begin{equation}
\label{eq:2_12} 
g_1=\frac{(u_{1/2}^+  + d_1^+ ) y_0} {\Delta _{1 /2}x} + f_1; \quad
g_i=f_i; \quad
g_N =  \frac{-(u_{N+1}^- + d_N^- )y_{N+1}} {\Delta _{N+1/2}} + f_N
\end{equation}

Elements of the matrices $A$ and $B^\pm$ are calculated as
\begin{align}
\label{eq:2_13} 
& a_{ii-1}  = \frac{ -u_{i-1/2}^+ - d_{i}^+} {\Delta _{i-1/2}x}; \quad &
& b_{i}^+  = r_i^+ \; \frac{y_i -y_{i-1} }{\Delta _{i-1/2}x} \notag \\
& a_{ii+1}  = \frac{u_{i+1/2}^- +d_{i}^- }{{\Delta _{i+1/2}x}}; \quad &
& b_{i}^-  = r_i^- \; \frac{y_{i+1} - y_i }{\Delta _{i+1/2}x} \\
& a_{ii}  = -a_{ii-1} -a_{ii+1}; \quad &
&  \notag
\end{align}


\section{Monotone difference scheme}
\label{Sec2}
We consider the system of equations \eqref{eq:2_11} for ${\boldsymbol{\alpha} ^{\pm,n}},{\boldsymbol{\alpha} ^{\pm,n+1}} = 0$
\begin{equation}
\label{eq:3_1}  
\begin{split} 
& \left[ {E + \Delta t\;\sigma {A^{n + 1}}+ \Delta t\sigma \Lambda ^{n+1}} \right]{\boldsymbol{y}^{n + 1}} - \Delta t\;\sigma {\boldsymbol g}^{n+1} \\
= & \left[ {E - \Delta t\;(1 - \sigma ){A^n} - \Delta t (1 -\sigma) \Lambda ^n} \right]{\boldsymbol{y}^n} + \Delta t\;(1 - \sigma ){\boldsymbol g}^{n}	
\end{split}
\end{equation}

In this section, we obtain the monotonicity condition for the difference scheme \eqref{eq:3_1} and derive for it an analog of the discrete local maximum principle, which plays a key role in the flux correction design.

\begin{defn}[\cite{Harten}] 
A difference scheme 
\begin{equation}
\label{eq:3_2} 
y_i^{n+1} = H(y_{i-k}^n,y_{i-k+1}^n,...,y_i^n,...,y_{i+l}^n)
\end{equation}
is said to be monotone if H is a monotone increasing function of each of its arguments.
\end{defn} 

\begin{theorem}
\label{th:3_1}
 If $\Delta t$ satisfies
\begin{equation}
\label{eq:3_3} 
\Delta t\sigma \mathop {\min }\limits_{1 \le i \le N} \lambda _i^{n+1} < 1
\end{equation}
\begin{equation}
\label{eq:3_4} 
\Delta t(1 - \sigma )\mathop {\max }\limits_{1 \le i \le N} \left[ {\frac{u_{i-1/2}^{+ ,n} + d_i^{+ ,n}}{\Delta _{i-1/2}x} - \frac{u_{i+1/2}^{- ,n} + d_i^{- ,n}}{\Delta _{i+1/2}x} + \lambda _i^n} \right] \le 1
\end{equation}
then the difference scheme \eqref{eq:3_1} is monotone.
\end{theorem}
\begin{proof}
If \eqref{eq:3_3} holds, the matrix $\left[ E + \Delta t\sigma (A^{n+1}+ \Lambda ^{n+1}) \right]$ is a strictly row diagonally dominant M-matrix. Then the inverse matrix $\left[ E + \Delta t\sigma( A^{n+1}+ \Lambda ^{n+1}) \right]^{-1}$ is a matrix with nonnegative elements.

The nonnegativity of the elements of matrix $\left[ E + \Delta t\sigma (A^{n+1}+ \Lambda ^{n+1}) \right]^{- 1}$ $\times \left[ E - \Delta t(1 - \sigma )( A^n + \Lambda ^n) \right]$ and, hence, the monotonicity of the scheme \eqref{eq:3_1} follows from the nonnegativity of the elements $\left[ {E - \Delta t(1 - \sigma )(A^n +\Lambda ^n} \right]$ for $\Delta t$ satisfying \eqref{eq:3_4}.
\end{proof}

\begin{theorem}
\label{th:3_2}
If $\Delta t$ satisfies 
\begin{equation}
\label{eq:3_5}
\Delta t(1 - \sigma )\mathop {\max }\limits_{1 \le i \le N} \left[ \frac{u_{i-1/2}^{+ ,n} + d_i^{+ ,n}}{\Delta _{i-1/2}x} - \frac{u_{i+1/2}^{- ,n} + d_i^{- ,n}}{\Delta _{i+1/2}x} \right] \le 1,	
\end{equation}
then the numerical solution of the system of equations \eqref{eq:3_1} satisfies the following inequalities   
\begin{equation}
\label{eq:3_6}
\begin{split}
& \mathop {\min }\limits_{k \in S_i} y_k^n - \Delta t(1 - \sigma )\lambda _i^n y_i^n + \Delta t(1 - \sigma )f_i^n  \\ 
\le & y_i^{n+1} + \Delta t \sigma \sum\limits_j {a_{ij}^{n+1}y_j^{n+1}} +\Delta t \sigma \lambda_i^{n+1} y_i^{n+1} - \Delta t \sigma g_i^{n+1} \\
\le & \mathop {\max }\limits_{k \in S_i} y_k^n - \Delta t(1 - \sigma )\lambda _i^n y_i^n + \Delta t(1 - \sigma )f_i^n
\end{split} 
\end{equation} 		
where ${S_i}$ is the stencil of the difference scheme \eqref{eq:3_1} for an $i$th grid node.
\end{theorem}
\begin{proof} 
Let us prove the right-hand side of inequality \eqref{eq:3_6}.
 We rewrite the $i$th row of the system of equations \eqref{eq:3_1} in the form
\begin{equation}
\label{eq:3_7}
\begin{split}
& y_i^{n+1} + \Delta t \sigma \sum\limits_j {a_{ij}^{n+1}y_j^{n+1}} +\Delta t \sigma \lambda_i^{n+1} y_i^{n+1} - \Delta t \sigma g_i^{n+1} \\
 = & \left[ 1 + \Delta t (1 -\sigma) \left( \frac{ u_{i+1/2}^{- ,n} + d_i^{- ,n} }{\Delta _{i+1/2}x} - \frac{  u_{i-1/2}^{+ ,n} + d_i^{+ ,n}}{\Delta _{i-1/2}x} \right) \right] y_i^n -  \Delta t (1 -\sigma) \lambda _i^{n} y_i^n\\
+ & \Delta t (1 -\sigma) \left( \frac{ u_{i-1/2}^{+ ,n} + d_i^{+ ,n}}{\Delta _{i-1/2}x} y_{i-1}^{n} - \frac{ u_{i+1/2}^{- ,n} + d_i^{- ,n} }{\Delta _{i+1/2}x} y_{i+1}^{n} \right)
+  \Delta t (1 -\sigma) f_i^{n}  
\end{split} 
\end{equation} 		

Under condition \eqref{eq:3_5}, the first and third terms on the right-hand side of \eqref{eq:3_7} are a convex linear combination, therefore
\begin{equation}
\label{eq:3_8}
\begin{split}
& y_i^{n+1} + \Delta t \sigma \sum\limits_j {a_{ij}^{n+1}y_j^{n+1}} +\Delta t \sigma \lambda_i^{n+1} y_i^{n+1} - \Delta t \sigma g_i^{n+1} \\
\le & \mathop {\max }\limits_{k \in S_i} y_k^n - \Delta t(1 - \sigma )\lambda _i^n y_i^n + \Delta t(1 - \sigma )f_i^n
\end{split} 
\end{equation} 		
		
The lower bound \eqref{eq:3_6} is obtained in a similar way, which proves the theorem.
\end{proof}
\begin{remark} 
Under condition \eqref{eq:3_3}, the matrix $G=\left[ {E + \Delta t\;\sigma (A^{n+1} +\Lambda^{n+1})} \right]$ is a non-singular M-matrix, therefore $G^{-1}$ is a nonnegative and isotone matrix~\cite[p.52,\;2.4.3]{Ortega}, i.e. if $\boldsymbol x \preceq \boldsymbol y$, then $G^{-1}\boldsymbol x \preceq G^{-1}\boldsymbol y  $. Here $\preceq$ denotes the natural (component-wise) partial ordering on $R^N$, i.e. $\boldsymbol x \preceq \boldsymbol y$ if and only if $x_i \leq y_i$ for all $i$. Thus, the change of the vector $\boldsymbol y^{n+1}$ can be controlled by changing the right-hand side of the equation \eqref{eq:3_1}.

Inequalities \eqref{eq:3_6} hold for the right-hand side of \eqref{eq:3_1} and will be used to obtain restrictions on flux limiters in the scheme \eqref{eq:2_11}. We can consider \eqref{eq:3_6} as an analogue of discrete local maximum principle for the scheme \eqref{eq:3_1}. Note that to obtain restrictions \eqref{eq:3_6}, it is sufficient for us to know the numerical solution of \eqref{eq:3_1} at a previous time step. 
\end{remark} 


\section{Finding flux limiters} 
\label{Sec3}
To find fux limiters for scheme \eqref{eq:2_11}, we implement the approach proposed in~\cite{Kivva}. 
Our goal is to find maximal values of the flux limiters for which the solution of the difference scheme \eqref{eq:2_11} is similar to the solution of the monotone difference scheme \eqref{eq:3_1}. For this, we require that the difference scheme \eqref{eq:2_11} satisfies inequalities \eqref{eq:3_6}. Then finding the flux limiters can be considered as the following optimization problem
\begin{equation}
\label{eq:4_1}
\Im (\boldsymbol \alpha ^{\pm ,n},{\boldsymbol \alpha ^{\pm ,n+1}}) = \sum\limits_{k = n}^{n + 1} {\sum\limits_{i = 1}^N {\alpha _i^{ + ,k}} }  + \sum\limits_{k = n}^{n + 1} {\sum\limits_{i = 1}^N {\alpha _i^{ - ,k}} }  \to \mathop {\max }\limits_{{\boldsymbol\alpha ^{ \pm ,n}},{\boldsymbol\alpha ^{\pm ,n + 1}} \in U_{ad}} 
\end{equation}
subject to \eqref{eq:2_11} and
\begin{equation}
\label{eq:4_2} 
\begin{split} 
&  {\underline {\boldsymbol y}^n} + \Delta t(1 - \sigma )\boldsymbol {f^n} \\
\le \left[ {E - \Delta t(1 - \sigma ){A^n}} \right]{\boldsymbol y^n} + & \Delta t{\left( {{B^ + }{\boldsymbol \alpha ^ + } + {B^ - }{\boldsymbol \alpha ^ - }} \right)^{(\sigma )}} + \Delta t(1 - \sigma ){\boldsymbol g^n}  \\
 \le &  {\bar {\boldsymbol y}^n} + \Delta t(1 - \sigma ){\boldsymbol f^n}
\end{split}  								
\end{equation}
where $\underline {\boldsymbol y}$ and $\bar {\boldsymbol y}$ are column vectors whose components are ${\underline y} {_i} = \mathop {\min }\limits_{j \in {S_i}} y_j$ and ${\bar y_i} = \mathop {\max }\limits_{j \in {S_i}} y_j$.
$U_{ad}$ is the set of vectors ${\boldsymbol{\alpha} ^{\pm,n}},{\boldsymbol{\alpha} ^{\pm,n+1}}$, which is defined as the Cartesian product of $N$-vectors
\begin{equation}
\label{eq:4_3}
  U^{ad}= \left\lbrace {\left( {\boldsymbol{\alpha}^{\pm,n}, \boldsymbol{\alpha}^{\pm,n+1} } \right) \in \left( R^N \right)^4 : \quad 0\leq \alpha_i^{\pm,k}\leq 1, \;\; k=n,n+1  }\right\rbrace 
\end{equation}
 
Note that for $\sigma  = 0$ the optimization problem \eqref{eq:4_1}-\eqref{eq:4_3} and \eqref{eq:2_11} is a linear programming problem, and for $\sigma  > 0$ it is a nonlinear programming problem.

To solve the nonlinear optimization problem \eqref{eq:4_1}-\eqref{eq:4_3} and \eqref{eq:2_11} in one time step, we use the following iterative process: 
\begin{enumerate}[\bfseries   Step 1.]
\item \label {it:1}
 Initialize positive numbers $\delta ,{\varepsilon _1},{\varepsilon _2} > 0$. Set $p = 0$, ${\boldsymbol{y}^{n+1,0}} = {\boldsymbol{y}^n}$, ${\boldsymbol{\alpha} ^{\pm,n,0}},{\boldsymbol{\alpha} ^{\pm,n+1,0}} = 0$.
\item \label {it:2}
 Find the solution ${\boldsymbol{\alpha} ^{\pm,n,p+1}},{\boldsymbol{\alpha} ^{\pm,n+1,p+1}}$ of the following linear programming problem
\begin{equation}
\label{eq:4_4}
\Im (\boldsymbol \alpha ^{\pm ,n,p+1},{\boldsymbol \alpha ^{\pm ,n+1,p+1}})   \to \mathop {\max }\limits_{{\boldsymbol\alpha ^{ \pm ,n,p+1}},{\boldsymbol\alpha ^{\pm, n+1, p+1}} \in U_{ad}} 			
\end{equation}
\begin{equation}
\label{eq:4_5}
\begin{split} 
& \; \mathop {\min }\limits_{j \in {S_i}} y_j^n - y_i^n + \Delta t\,(1 - \sigma )  \sum\limits_{j \ne i} {a_{ij}^n} \left( {y_j^n - y_i^n} \right) \\
 \le & \; \Delta t \, (1 - \sigma ) \left( b_i^{+,n} {\alpha _i ^{+,n,p+1}} +  b_i^{-,n} {\alpha _i ^{-,n,p+1}} \right) \\
+ & \; \Delta t\,\sigma \left( b_i^{+,n+1,p}{\alpha _i ^{+,n+1,p+1}} + b_i^{-,n+1,p}{\alpha _i ^{-,n+1,p+1}} \right)  \\
\le & \; \mathop {\max }\limits_{j \in {S_i}} y_j^n - y_i^n + \Delta t\,(1 - \sigma )  \sum\limits_{j \ne i} {a_{ij}^n} \left( {y_j^n - y_i^n} \right)
\end{split}  
\end{equation}
\item \label {it:3}
 For the ${\boldsymbol{\alpha} ^{\pm,n,p+1}},{\boldsymbol{\alpha} ^{\pm,n+1,p+1}}$, find $y_i^{n+1,p+1}$ from the system of linear equations
\begin{equation}
\label{eq:4_6}
\begin{split} 
& \left[ {E + \Delta t \sigma \left( A^{n+1} + \Lambda ^{n+1} \right) } \right]{\boldsymbol{y}^{n+1,p+1}}  = \left[ {E - \Delta t (1 - \sigma )\left( A^n + \Lambda^n  \right)} \right]{\boldsymbol{y}^n} \\ 
& + \Delta t\, \left[ \left( B^{+,p}{\boldsymbol{\alpha} ^{+,p+1}} \right) ^{(\sigma)} + \left( B^{-,p}{\boldsymbol{\alpha} ^{-,p+1}} \right) ^{(\sigma)} \right] + \Delta t\,{\boldsymbol{g}^{(\sigma )}}  
\end{split} 	
\end{equation}
\item \label {it:4}
 Algorithm stop criterion
\begin{equation}
\label{eq:4_7}
\begin{split} 
& \mathop {\max} \limits_i
\frac{{\left| {y_i^{n + 1,p + 1} - y_i^{n + 1,p}} \right|}}{{\max \left( {\delta ,\left| {y_i^{n + 1,p + 1}} \right|} \right)}} < {\varepsilon _1},   \\
& \left|  \Im \left( \boldsymbol \alpha^{\pm,n,p+1}, \boldsymbol \alpha^{\pm,n+1,p+1} \right) -  \Im \left( \boldsymbol \alpha^{\pm,n,p}, \boldsymbol \alpha^{\pm,n+1,p} \right) \right| < {\varepsilon _2} 		
\end{split} 	
\end{equation}
If conditions \eqref{eq:4_7} hold, then ${\boldsymbol{y}^{n+1}} = {\boldsymbol{y}^{n+1,p+1}}$. Otherwise, set $p = p+1$ and go to {\bf Step \ref{it:2}}.
 
\end{enumerate} 

The solvability of the linear programming problem \eqref{eq:4_4}-\eqref{eq:4_5} is considered in the theorem below. 

\begin{theorem}
\label{th:4_1}
Assume that $\Delta t \,$ satisfies \eqref{eq:3_3}-\eqref{eq:3_5}, then the linear programming problem \eqref{eq:4_4}-\eqref{eq:4_5} is solvable.
\end{theorem} 
\begin{proof}
 To prove that problem \eqref{eq:4_4}-\eqref{eq:4_5} is solvable, it is sufficient to show that the objective function $\Im ({\boldsymbol{\alpha} ^{\pm,n}},{\boldsymbol{\alpha} ^{\pm,n+1}})$ is bounded and the feasible set is non-empty. The boundedness of the function \eqref{eq:4_1} follows from the boundedness of the vectors ${\boldsymbol{\alpha} ^{\pm,n}}$ and ${\boldsymbol{\alpha} ^{\pm,n+1}}$ whose coordinates vary from zero to one. On the other hand, if the hypothesis of the theorem is true, then the zero vectors ${\boldsymbol{\alpha} ^{pm,n}}$ and ${\boldsymbol{\alpha} ^{\pm,n+1}}$ satisfy the system of inequalities \eqref{eq:4_5}.
 
 This completes the proof of the theorem. 
\end{proof}

\section{Flux limiter design}
\label{Sec4}
In the iterative process described in the previous section, the flux limiters are found by solving the linear programming problem \eqref{eq:4_4}-\eqref{eq:4_5}. Solving a linear programming problem requires additional computational cost. Therefore, in the iterative process at {\bf Step \ref{it:2}}, instead of \eqref{eq:4_4}-\eqref{eq:4_5}, we use its approximate solution.

The purpose of this section is to find a nontrivial approximate solution to the linear programming problem \eqref{eq:4_4}-\eqref{eq:4_5}.
Nonzero $\left( {\boldsymbol{\alpha} ^{\pm,n}},{\boldsymbol{\alpha} ^{\pm,n+1}} \right) \in {U_{ad}}$ satisfy the system of inequalities \eqref{eq:4_5}, and, omitting the iteration number, we rewrite the latter in the form 
\begin{equation}
\label{eq:5_1} 
\begin{split} 
 (1 - \sigma ) & \left( b_i^{+,n} {\alpha _i ^{+,n}} +  b_i^{-,n} {\alpha _i ^{-,n}} \right) 
+ \sigma \left( b_i^{+,n+1}{\alpha _i ^{+,n+1}} + b_i^{-,n+1}{\alpha _i ^{-,n+1}} \right)  \\
& \le \; \frac{1}{\Delta t} \left(  \mathop {\max }\limits_{j \in {S_i}} y_j^n - y_i^n \right) + (1 - \sigma )  \sum\limits_{j \ne i} {a_{ij}^n} \left( {y_j^n - y_i^n} \right)
\end{split}
\end{equation}
\begin{equation}
\label{eq:5_2} 
\begin{split} 
 (1 - \sigma ) & \left( b_i^{+,n} {\alpha _i ^{+,n}} +  b_i^{-,n} {\alpha _i ^{-,n}} \right) 
+ \sigma \left( b_i^{+,n+1}{\alpha _i ^{+,n+1}} + b_i^{-,n+1}{\alpha _i ^{-,n+1}} \right)  \\
&  \ge \; \frac{1}{\Delta t} \left( \mathop {\min }\limits_{j \in {S_i}} y_j^n - y_i^n \right) + (1 - \sigma )  \sum\limits_{j \ne i} {a_{ij}^n} \left( {y_j^n - y_i^n} \right)
\end{split}
\end{equation}
\begin{equation}
\label{eq:5_3}
 0 \le \alpha _i^{\pm,n} \le 1, \quad  0 \le \alpha _i^{\pm,n+1} \le 1		   						
\end{equation}

For the left-hand sides of inequalities \eqref{eq:5_1} and \eqref{eq:5_2}, the following estimates are valid
\begin{equation}
\begin{split}
\label{eq:5_4}  
(1 - \sigma ) &  \left( b_i^{+,n} {\alpha _i ^{+,n}} +  b_i^{-,n} {\alpha _i ^{-,n}} \right) 
+ \sigma \left( b_i^{+,n+1}{\alpha _i ^{+,n+1}} + b_i^{-,n+1}{\alpha _i ^{-,n+1}} \right)  \\
& \le \alpha _i^{+,max} \left[ (1-\sigma) \left( \max (0,b_i^{+,n}) + \max (0,b_i^{-,n}) \right) \right. \\
& \qquad \quad + \left. \sigma \left( \max (0,b_i^{+,n+1}) + \max (0,b_i^{-,n+1}) \right) \right]
\end{split} 
\end{equation}
\begin{equation}
\begin{split}
\label{eq:5_5}  
(1 - \sigma ) &  \left( b_i^{+,n} {\alpha _i ^{+,n}} +  b_i^{-,n} {\alpha _i ^{-,n}} \right) 
+ \sigma \left( b_i^{+,n+1}{\alpha _i ^{+,n+1}} + b_i^{-,n+1}{\alpha _i ^{-,n+1}} \right)  \\
& \ge \alpha _i^{-,max} \left[ (1-\sigma) \left( \min (0,b_i^{+,n}) + \min (0,b_i^{-,n}) \right) \right. \\
& \qquad \quad + \left. \sigma \left( \min (0,b_i^{+,n+1}) + \min (0,b_i^{-,n+1}) \right) \right]
\end{split} 
\end{equation}
where $ \alpha _i^{+,max}$ and $ \alpha _i^{-,max}$ are the maximums of the components $\alpha_i^{\pm,n}$ and $\alpha _i^{\pm,n+1}$ corresponding to the non-negative and non-positive coefficients $b_i^{\pm}$ on the left-hand sides of \eqref{eq:5_4} and \eqref{eq:5_5}, respectively. 
 
  Substituting \eqref{eq:5_4} into \eqref{eq:5_1}, and \eqref{eq:5_5} into \eqref{eq:5_2} yields
\begin{equation}
\label{eq:5_6}
 \alpha _i^{\pm,k} = \left\{ {\begin{array}{*{20}{c}} {R_i^ +  \qquad {\rm{    }}b_i^{\pm,k} > 0} \\
{R_i^ -  \qquad {\rm{    }}b_i^{\pm,k} < 0}\end{array}} \right. \qquad k=n,n+1								
\end{equation}
where 
\begin{equation}
\label{eq:5_7}
 R_i^ \pm  = \min \left( 1, \alpha _i^{\pm,max} \right) =\min \left(1,{Q_i^ \pm }/{P_i^ \pm } \right) 					
\end{equation}
\begin{equation}
\label{eq:5_8}
 Q_i^ +  = \frac{1}{{\Delta t}}\left( {\mathop {\max }\limits_{j \in {S_i}} y_j^n - y_i^n} \right) + (1 - \sigma )\sum\limits_{j \ne i} {a_{ij}^n} \left( {y_j^n - y_i^n} \right)  						
\end{equation}
\begin{equation}
\label{eq:5_9}
 Q_i^ -  = \frac{1}{{\Delta t}}\left( {\mathop {\min }\limits_{j \in {S_i}} y_j^n - y_i^n} \right) + (1 - \sigma )\sum\limits_{j \ne i} {a_{ij}^n} \left( {y_j^n - y_i^n} \right) 						
\end{equation}
\begin{equation}
\label{eq:5_10} 
\begin{split} 
 P_i^+ & = (1-\sigma) \left( \max (0,b_i^{+,n}) + \max (0,b_i^{-,n}) \right) \\
& +\sigma \left( \max (0,b_i^{+,n+1}) + \max (0,b_i^{-,n+1}) \right)   
\end{split} 			
\end{equation}
\begin{equation}
\label{eq:5_11}
\begin{split} 
 P_i^ - & = (1-\sigma) \left( \min (0,b_i^{+,n}) + \min (0,b_i^{-,n}) \right)  \\
& + \sigma \left( \min (0,b_i^{+,n+1}) + \min (0,b_i^{-,n+1}) \right) 			
\end{split} 			
\end{equation}
\begin{remark}
Note that similarly, the flux correction formulas can be obtained  for the convex combination of \eqref{eq:2_2} and \eqref{eq:2_4}, which approximates the convective term in equation \eqref{eq:1_1}. This approach is also applicable for schemes with a high-order approximation of the convective-diffusive flux. Moreover, this method and formulas \eqref{eq:5_6}-\eqref{eq:5_11} can be easily generalized to the multidimensional case.
\end{remark} 

 
\section{Numerical Results} \label{Sec5}
We conclude the paper with a number of numerical tests. The purpose of this section is to compare the results of the difference schemes considered in the paper. Below, we abbreviate by NDVL and NDVA the difference scheme \eqref{eq:2_11}, flux limiters of which are exact or approximate solutions of the linear programming problem \eqref{eq:4_4}-\eqref{eq:4_5}. We also use DIV notation for the difference scheme, the flux correction of which is based on the divergent part of the convective flux \eqref{eq:2_5_4}.

In our calculations, we apply the GLPK (GNU Linear Programming Kit) v.4.65 set of routines for solving linear programming, mixed integer programming, and other related problem. GLPK is available at \href{url}{https://www.gnu.org/software/glpk/}.


\subsection{One-Dimensional Advection}  \label{Sec5_1}

\begin{figure*}[!b]
\centering
  \includegraphics[width=3.9cm]{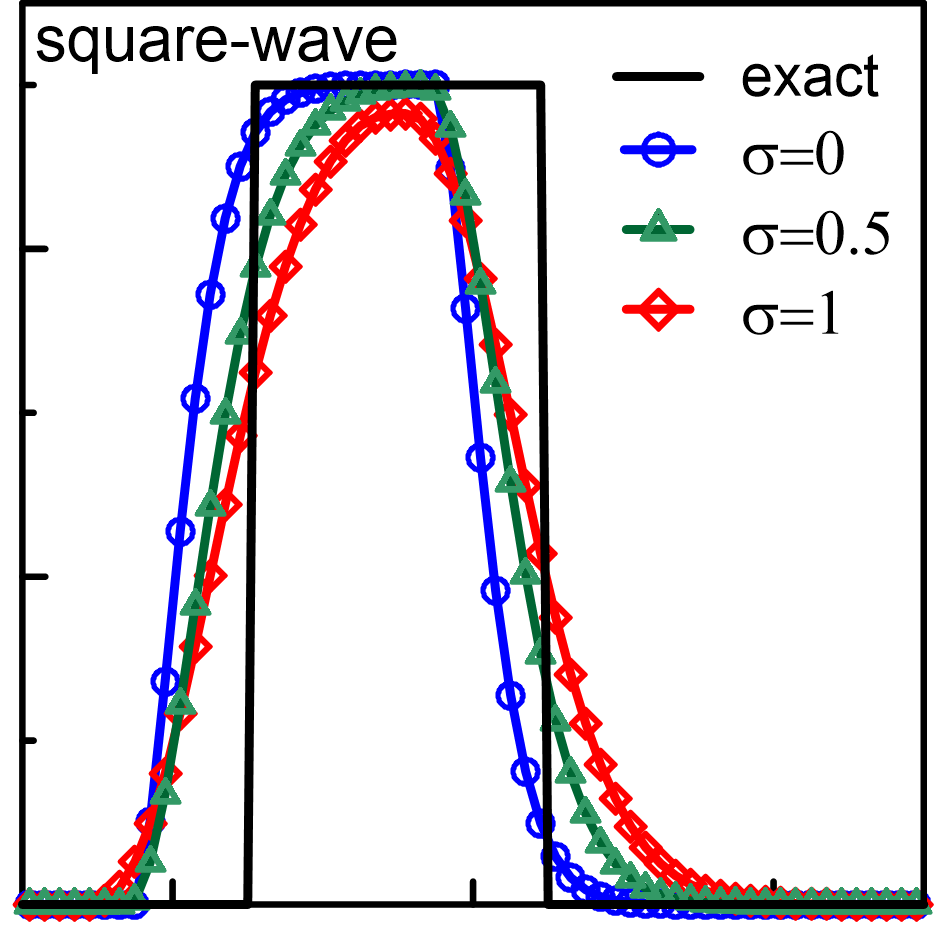}  
  \includegraphics[width=3.9cm]{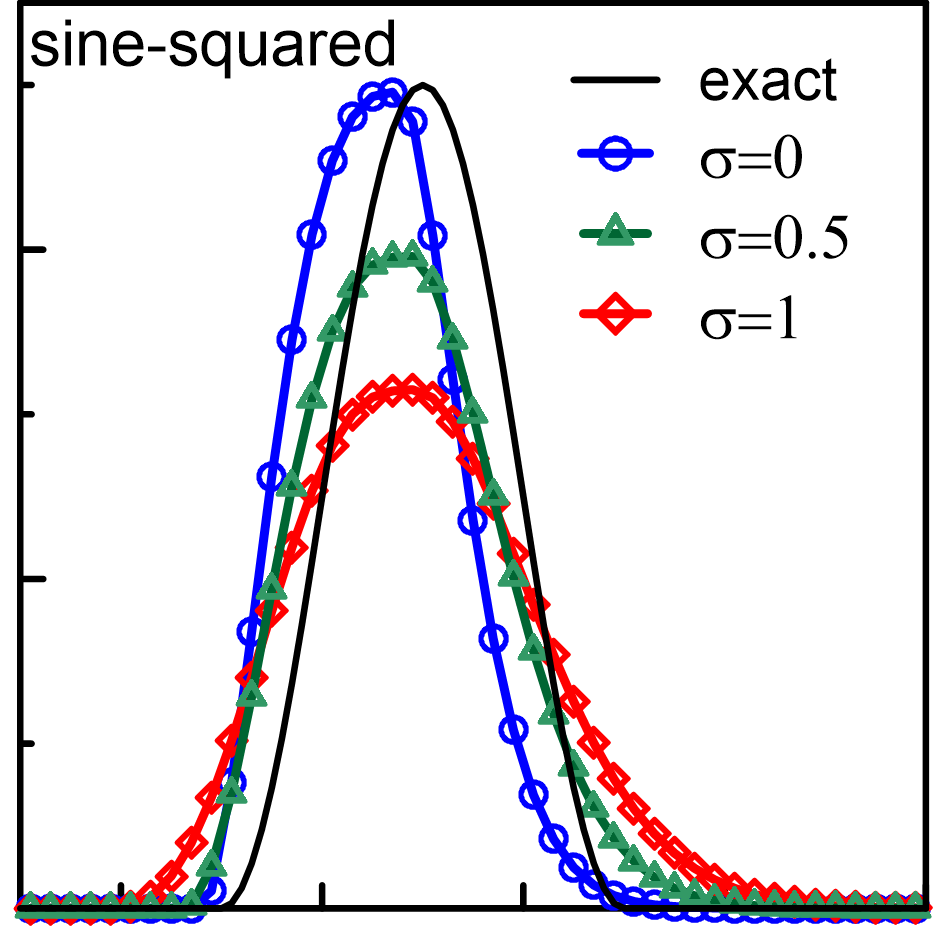}  
  \includegraphics[width=3.9cm]{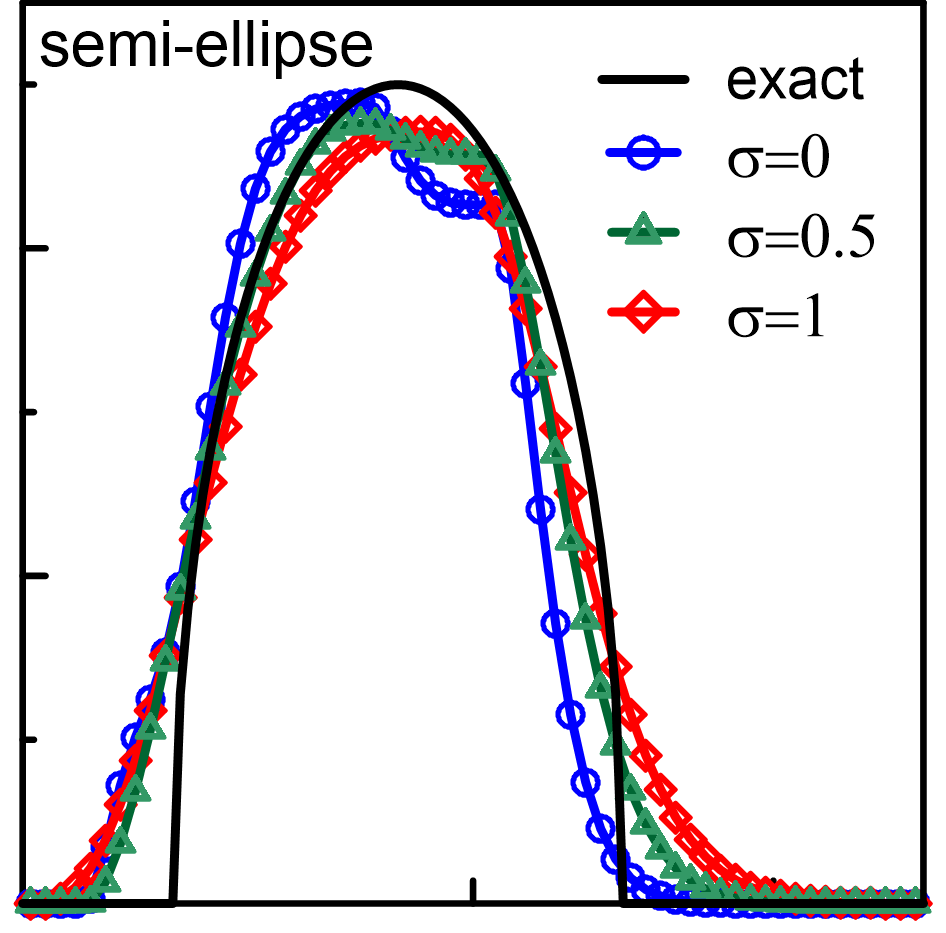} 
\begin{tabular}[t]{cc} 
  \includegraphics[width=3.9cm]{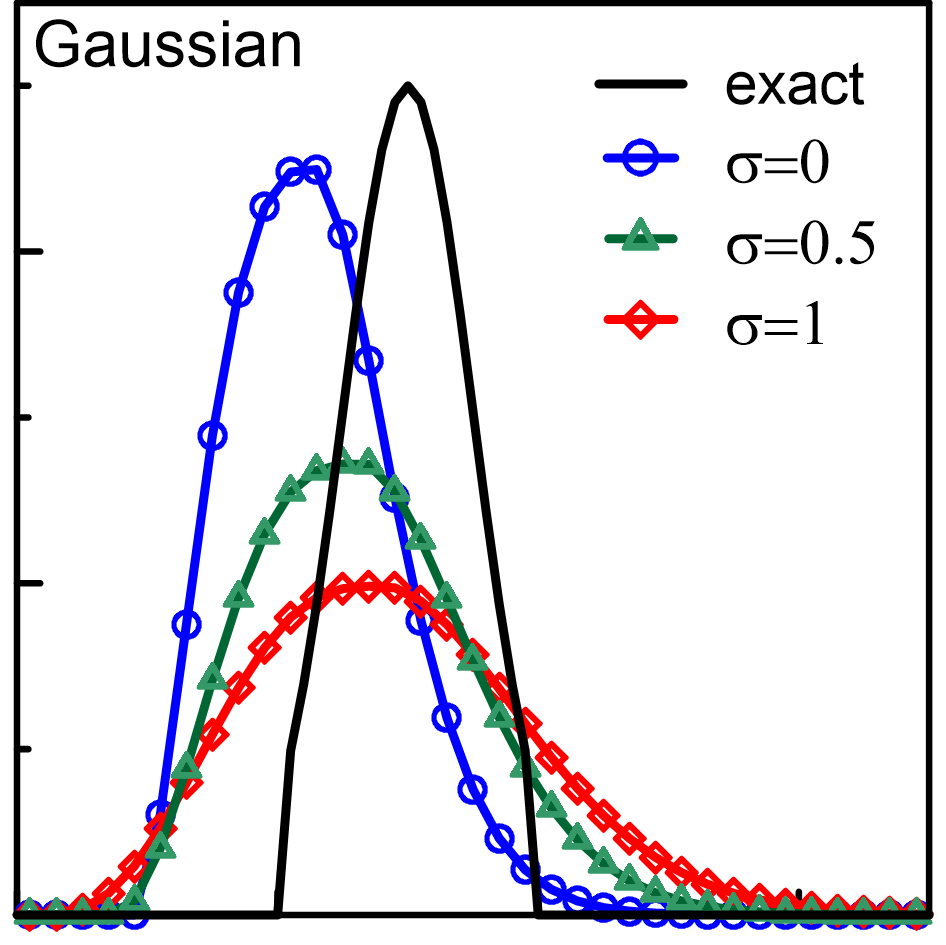}  &
  \includegraphics[width=3.9cm]{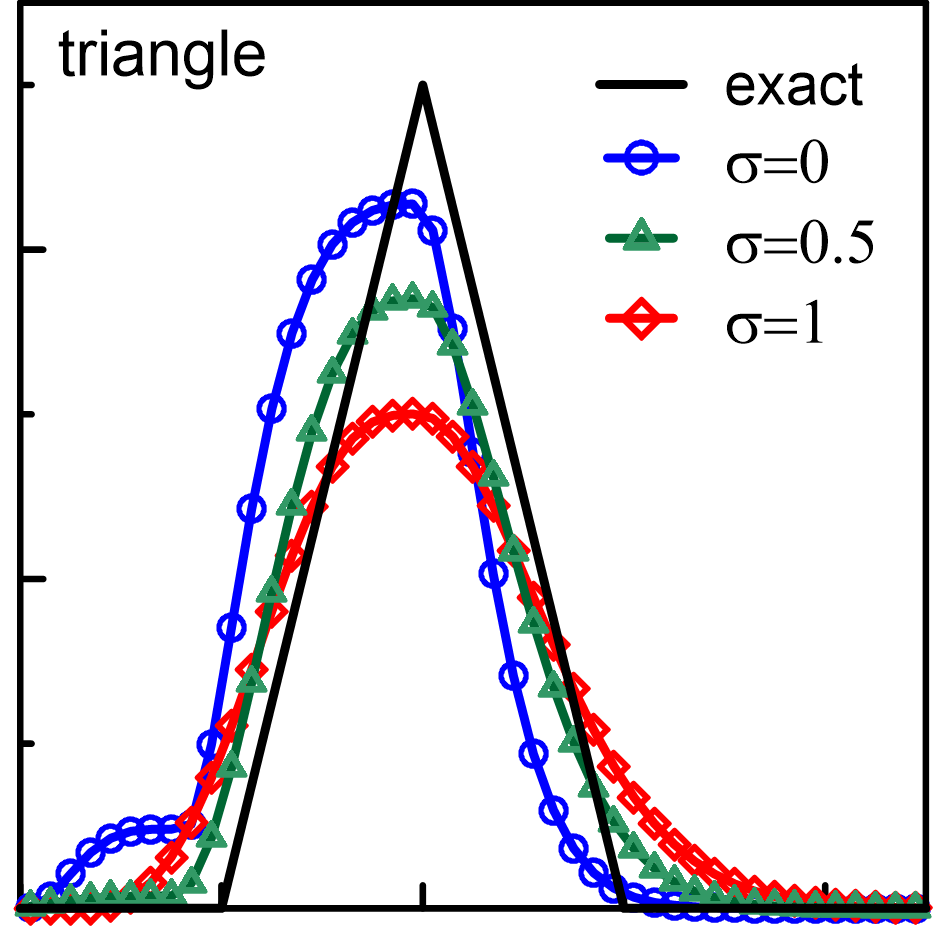}
\end{tabular}   
\caption{Numerical results for the advection test \eqref{eq:6_1} with the NDVL scheme for various weights $\sigma $ .  Flux limiters are calculated using the linear programming problem  \eqref{eq:4_4}-\eqref{eq:4_5}}
\label{fig:1}       
\end{figure*}

\begin{figure*}[!t]
\centering
  \includegraphics[width=3.9cm]{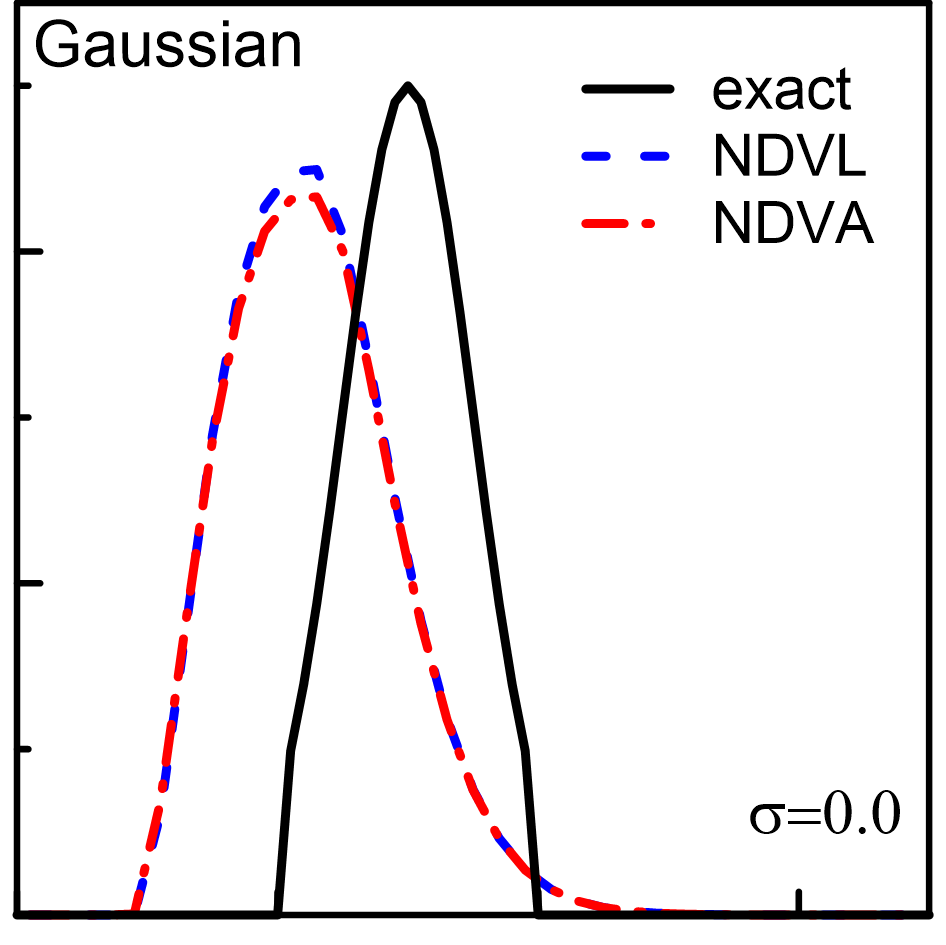}  
  \includegraphics[width=3.9cm]{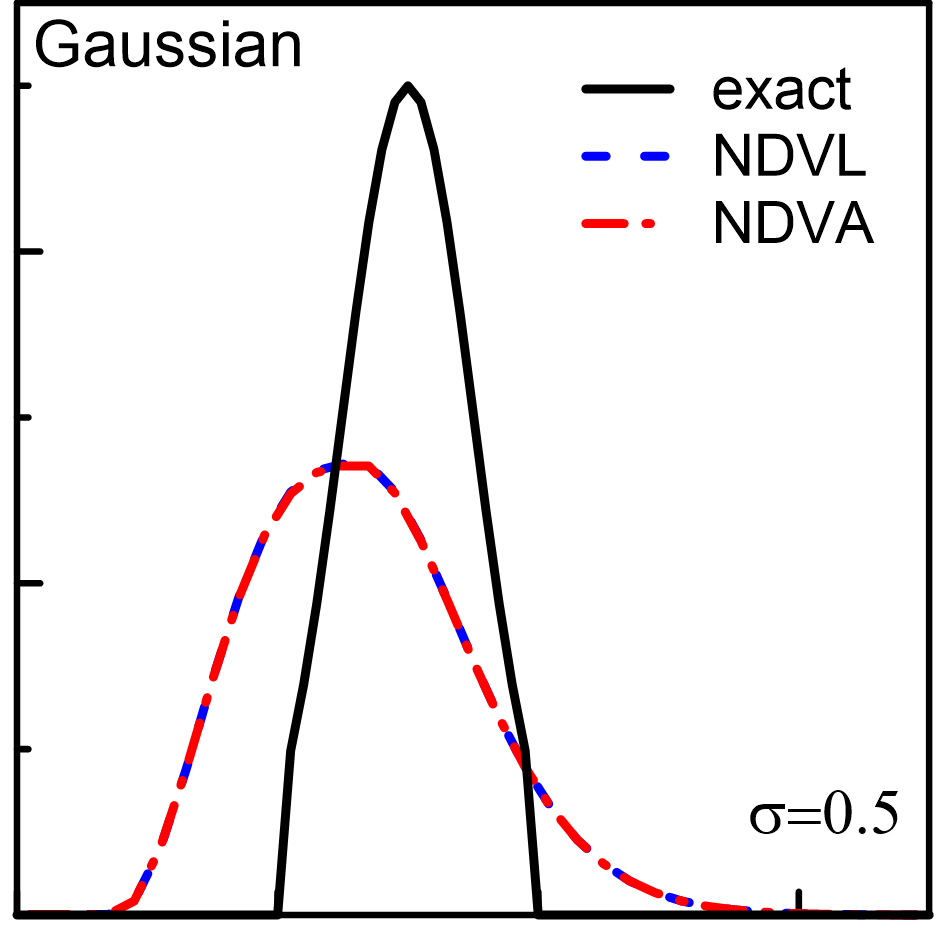}  
  \includegraphics[width=3.9cm]{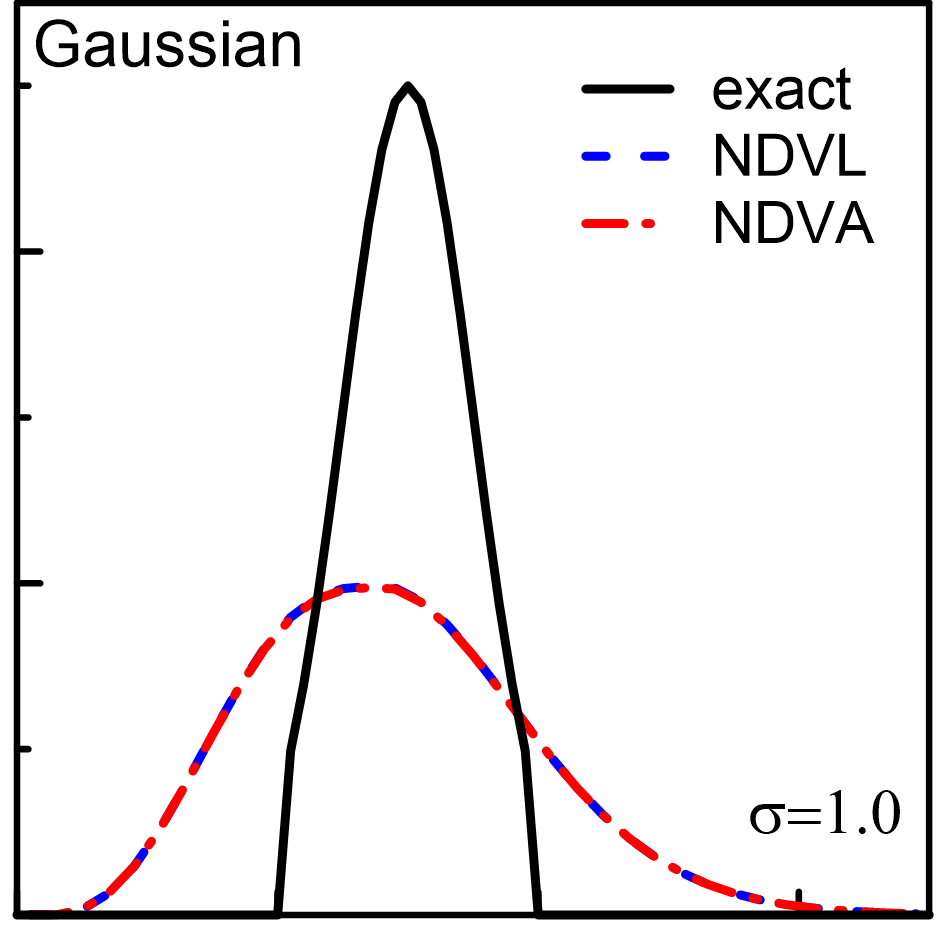} 
\caption{Comparison of the results for the advection test \eqref{eq:6_1} with the NDVL and NDVA schemes for $\sigma =0.5 $. }
\label{fig:2}       
\end{figure*}

We consider the one-dimensional advection test of Leonard et al. \cite{b31} on the uniform grid with $\Delta x = 0.01$ and constant velocity. The initial scalar profile consists of five different shapes: square wave, sine-squared, semi-ellipse, Gaussian, and triangle. The initial profile is specified as
\begin{equation}
\label{eq:6_1}
 y({x_i}) =  \left\{ {\begin{array}{*{20}{l}}
1&{{\rm if} \;\; 0.05 \le {x_i} \le 0.25}&{{\rm{(square}}\;{\rm{wave)}}} \\
{{{\sin }^2}\left[ {\dfrac{\pi }{{0.2}}\left( {{x_i} - 0.85} \right)} \right]}&{{\rm if} \;\; 0.85 \le {x_i} \le 1.05}&{{\rm{(sine - squared)}}} \\
{\sqrt {1 - {{\left[ {\dfrac{1}{{15\Delta x}}\left( {{x_i} - 1.75} \right)} \right]}^2}} }& {{\rm if} \;\; 1.6 \le {x_i} \le 1.9}&{{\rm{(semi - ellipse)}}} \\
{\exp \left[ { - \dfrac{1}{{2{\gamma ^2}}}{{\left( {{x_i} - 2.65} \right)}^2}} \right]}& {{\rm if} \;\; 2.6 \le {x_i} \le 2.7}&{{\rm{(Gaussian)}}} \\
{10\left( {{x_i} - 3.3} \right)}& {{\rm if} \;\; 3.3 \le {x_i} \le 3.4}&{{\rm{(triangle)}}} \\
{1.0 - 10\left( {{x_i} - 3.4} \right)}& {{\rm if} \;\; 3.4 \le {x_i} \le 3.5}&{} \\
0&{{\rm{otherwise}}}&{}
\end{array}} \right.   				
\end{equation} 
The standard deviation for the Gaussian profile is specified as $\gamma  = 2.5$.

Numerical results with the NDVL scheme after 400 time steps at a Courant number of 0.2 are shown in Fig.~\ref{fig:1}. The flux limiters are calculated using the linear programming problem \eqref{eq:4_4}-\eqref{eq:4_5}. 
At the right edge of the semi-ellipse for $\sigma=0$ and $\sigma=0.5$, we observe the well-known "terracing" phenomenon, which is  a nonlinear effect of residual phase errors. It is shown in~\cite{b36},\cite{b3} that high-order FCT methods (above fourth-order) significantly reduce phase errors and that selective adding diffusion can also reduce terracing.
In the numerical solution of the implicit scheme, there is no terracing. The implicit scheme is more diffusive than the previous two, and its numerical solution is also more diffusive.

\begin{figure*}[!b]
\centering
  \includegraphics[width=3.9cm]{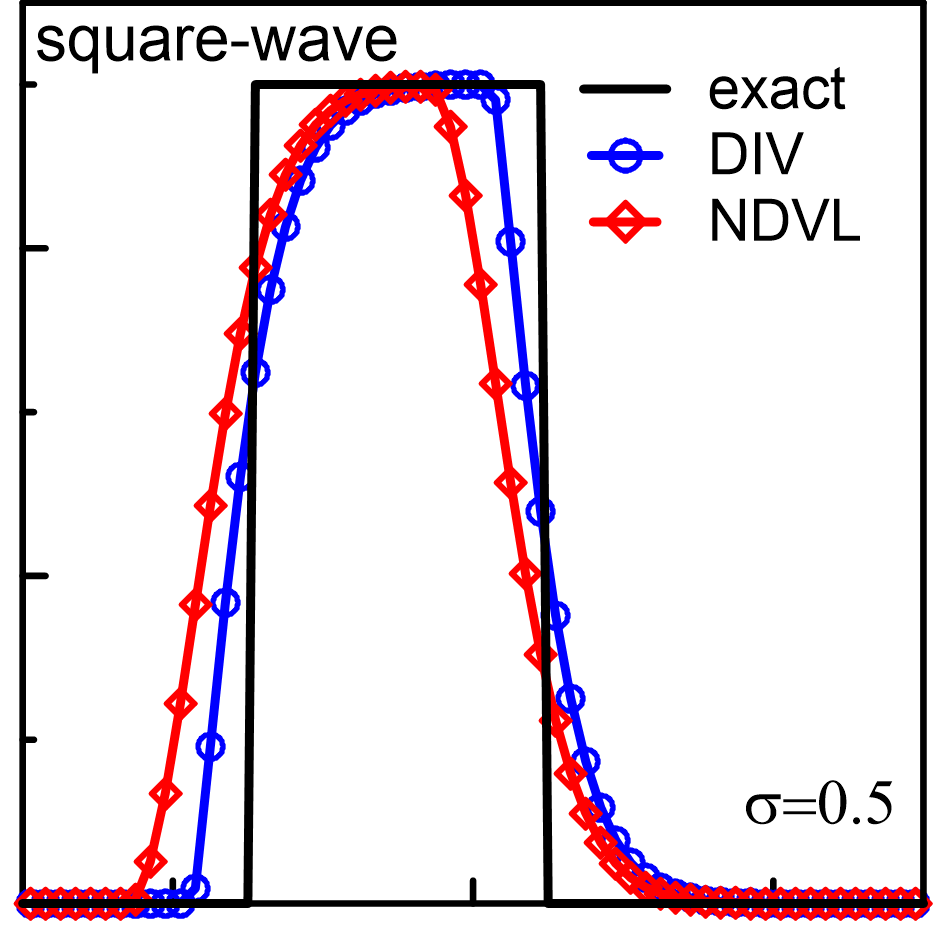} 
  \includegraphics[width=3.9cm]{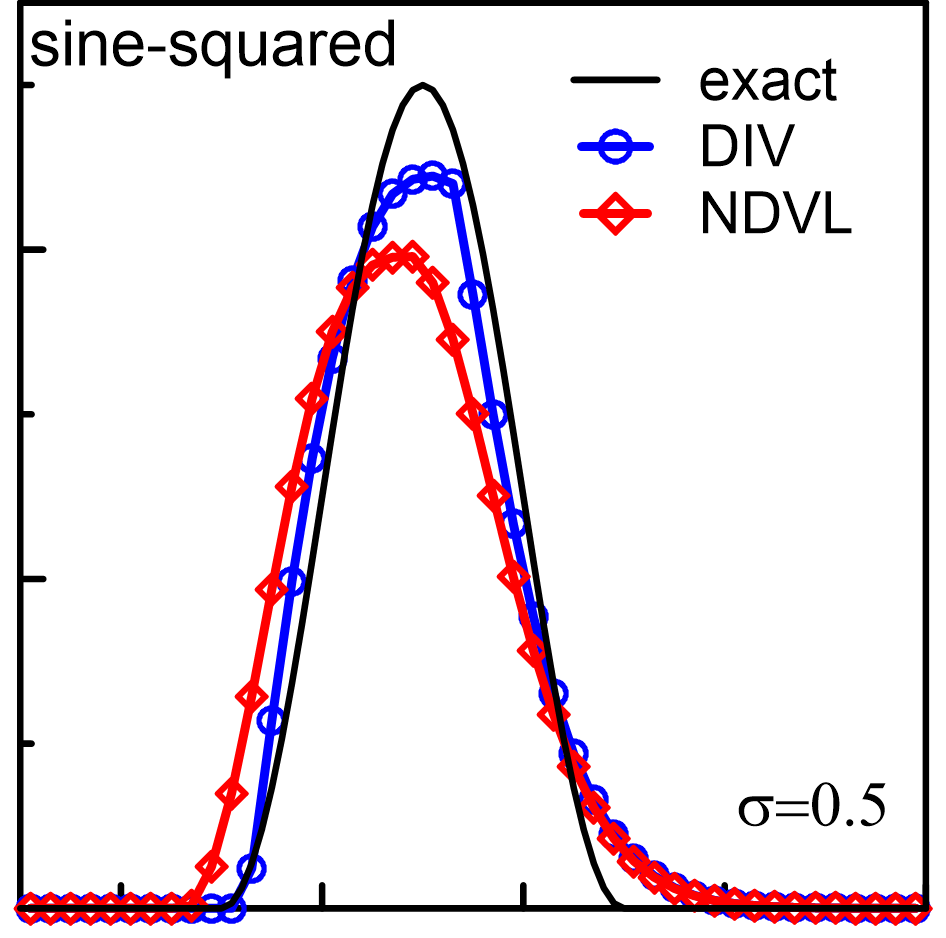} 
  \includegraphics[width=3.9cm]{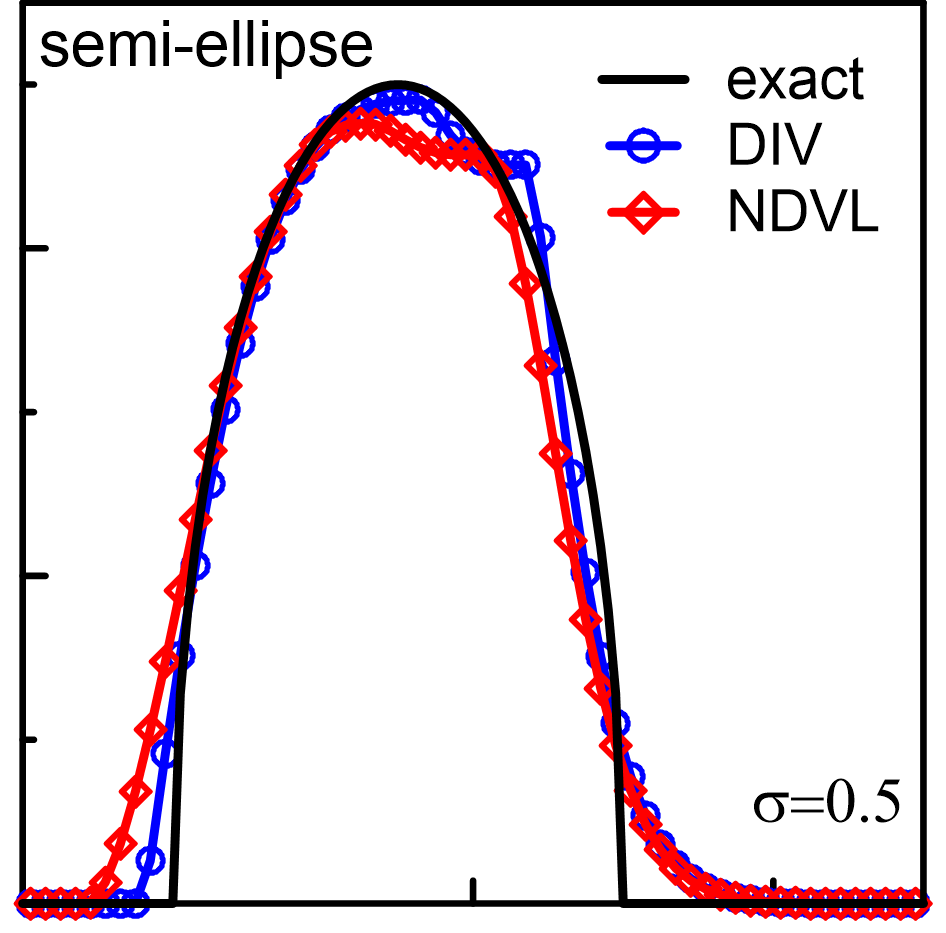}
\begin{tabular}[t]{cc} 
  \includegraphics[height=3.8cm]{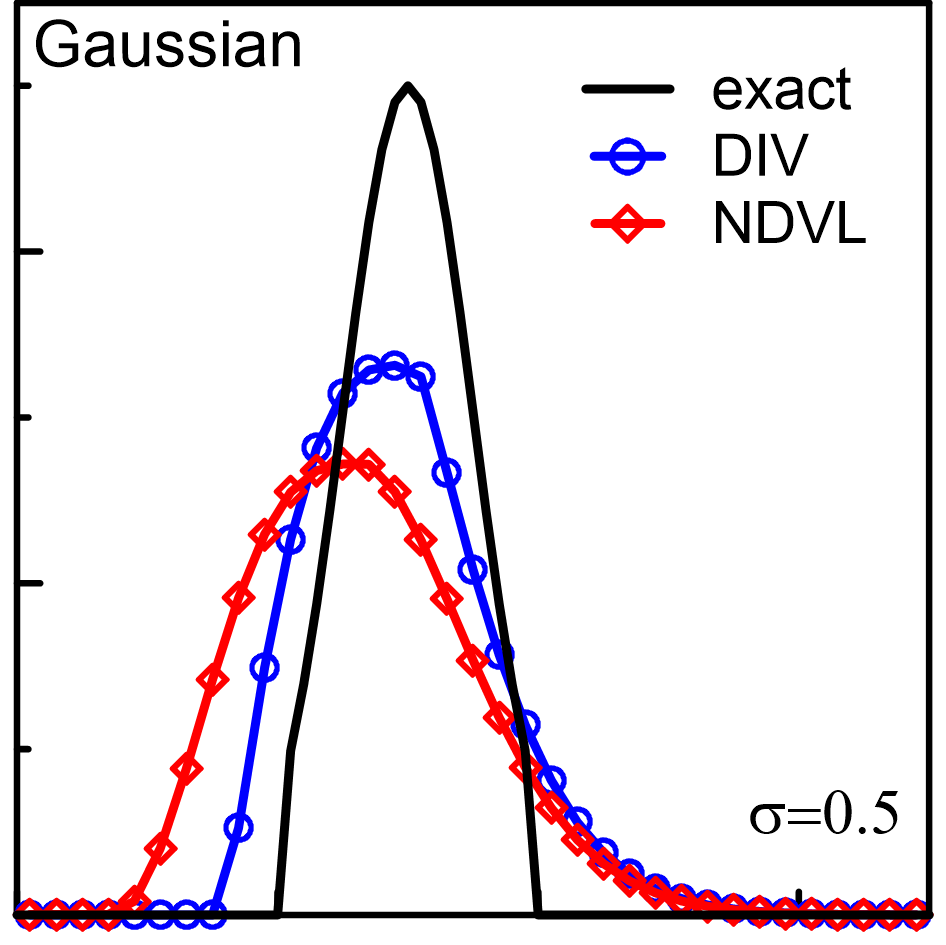} &
  \includegraphics[height=3.8cm]{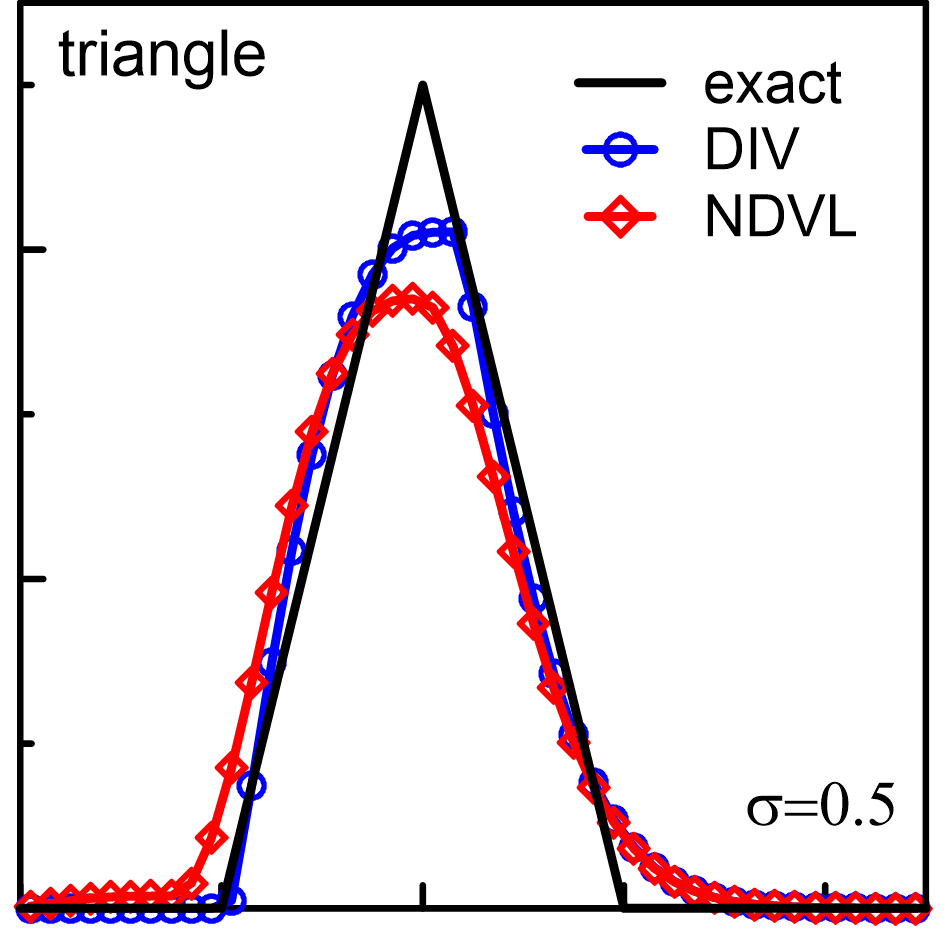} 
\end{tabular}   
\caption{Comparison of the numerical results for the advection test \eqref{eq:6_1} with the DIV and NDVL schemes for $\sigma=0.5$.}
\label{fig:3}       
\end{figure*}

\begin{table} [!t]
\caption{$L^1 $–norm of errors and the maximum values of the numerical results for the advection test \eqref{eq:6_1} with the DIV, NDVL and NDVA schemes.}
\label{tab1}       
\begin{tabular}{llllllll}
\hline\noalign{\smallskip}
&  &  \multicolumn{2}{c}{DIV} & \multicolumn{2}{c}{NDVL} & \multicolumn{2}{c}{NDVA}  \\
 \cline{3-4} \cline{5-6} \cline{7-8} \noalign{\smallskip}
 & $\sigma$  &  $L^1$ error & $y_{max}$ & $L^1$ error & $y_{max}$  & $L^1$ error & $y_{max}$   \\
\noalign{\smallskip}\hline\noalign{\smallskip}
\multirow{3}{*}{wav}  
	& 0.0 &	2.1811$\times 10^{-2}$ & 1.0000 & 8.1136$\times 10^{-2}$ & 1.0000 & 8.1182$\times 10^{-2}$ & 1.0000 \\
	& 0.5 &	4.3933$\times 10^{-2}$ & 0.9997 & 6.5511$\times 10^{-2}$ & 0.9976 & 6.5527$\times 10^{-2}$ & 0.9973 \\
	& 1.0 &	6.9477$\times 10^{-2}$ & 0.9843 & 7.6861$\times 10^{-2}$ & 0.9653 & 7.6774$\times 10^{-2}$ & 0.9650 \\
\hline	\noalign{\smallskip}
\multirow{3}{*}{sine}
	& 0.0 &	1.6883$\times 10^{-2}$ & 0.9938 & 4.6661$\times 10^{-2}$ & 0.9913 & 4.7052$\times 10^{-2}$ & 0.9766 \\
	& 0.5 &	1.6423$\times 10^{-2}$ & 0.8895 & 3.2650$\times 10^{-2}$ & 0.7917 & 3.2759$\times 10^{-2}$ & 0.7899 \\
	& 1.0 &	3.9029$\times 10^{-2}$ & 0.7043 & 4.5601$\times 10^{-2}$ & 0.6300 & 4.5694$\times 10^{-2}$ & 0.6286 \\
\hline	\noalign{\smallskip}
\multirow{3}{*}{elp}
	& 0.0 &	1.7926$\times 10^{-2}$ & 0.9973 & 4.9044$\times 10^{-2}$ & 0.9774 & 4.8959$\times 10^{-2}$ & 0.9775 \\
	& 0.5 &	1.7913$\times 10^{-2}$ & 0.9810 & 2.8675$\times 10^{-2}$ & 0.9526 & 2.8660$\times 10^{-2}$ & 0.9524 \\
	& 1.0 &	3.6078$\times 10^{-2}$ & 0.9601 & 3.9624$\times 10^{-2}$ & 0.9421 & 3.9603$\times 10^{-2}$ & 0.9422 \\
\hline	\noalign{\smallskip}
\multirow{3}{*}{gau}
	& 0.0 &	1.3639$\times 10^{-2}$ & 0.9764 & 6.9049$\times 10^{-2}$ & 0.8991 & 6.8116$\times 10^{-2}$ & 0.8661 \\
	& 0.5 &	2.7592$\times 10^{-2}$ & 0.6629 & 4.5303$\times 10^{-2}$ & 0.5438 & 4.5337$\times 10^{-2}$ & 0.5417 \\ 
	& 1.0 &	4.3681$\times 10^{-2}$ & 0.4828 & 4.8852$\times 10^{-2}$ & 0.3965 & 4.8882$\times 10^{-2}$ & 0.3949 \\
\hline	\noalign{\smallskip}
\multirow{3}{*}{tri}
	& 0.0 &	2.5205$\times 10^{-2}$ & 0.9389 & 4.8921$\times 10^{-2}$ & 0.8555 & 4.8870$\times 10^{-2}$ & 0.8517 \\
	& 0.5 &	1.3843$\times 10^{-2}$ & 0.8216 & 2.6126$\times 10^{-2}$ & 0.7404 & 2.6180$\times 10^{-2}$ & 0.7391 \\
	& 1.0 &	3.1245$\times 10^{-2}$ & 0.6655 & 3.7023$\times 10^{-2}$ & 0.6006 & 3.7123$\times 10^{-2}$ & 0.5991 \\
\noalign{\smallskip}\hline \noalign{\smallskip}
\end{tabular}
wav = Square wave; sine = Sine-squared; elp = Semi-ellipse;
gau = Gaussian; tri = Triangle.
\end{table}

The Gaussian test problem has a single moving maximum and shows the effects of “clipping” the solution. This is because the flux limiter cannot account for the true peak of the Gaussian as it passes between the grid points. The maximum is clipped less as the order of the algorithm increases. The key to good performance here is the application of a more flexible limiter and a more accurate estimate of the allowable upper and lower bounds on the solution~\cite{b2,b3}.

The numerical results for which the flux limiters are calculated  using exact and approximate solutions of the linear programming problem \eqref{eq:4_4}-\eqref{eq:4_5} are slightly different.  Their $L^1$-norm of  errors and the maximum values are presented in Table~\ref{tab1}. The comparison of the NDVL and NDVA results with $\sigma=0.5$ is given in Fig.~\ref{fig:2}.

In Fig.~\ref{fig:3} the solutions computed by the NDVL scheme are compared with the DIV scheme. 
Their $L^1$-norm of errors and the maximum values are presented in Table~\ref{tab1}.
Notice, that both the maximum values and the  errors of the DIV scheme are better than the corresponding maximum values and errors of the NDVL scheme.

\begin{figure}[!b]
  \centering 
  \includegraphics[scale=0.8]{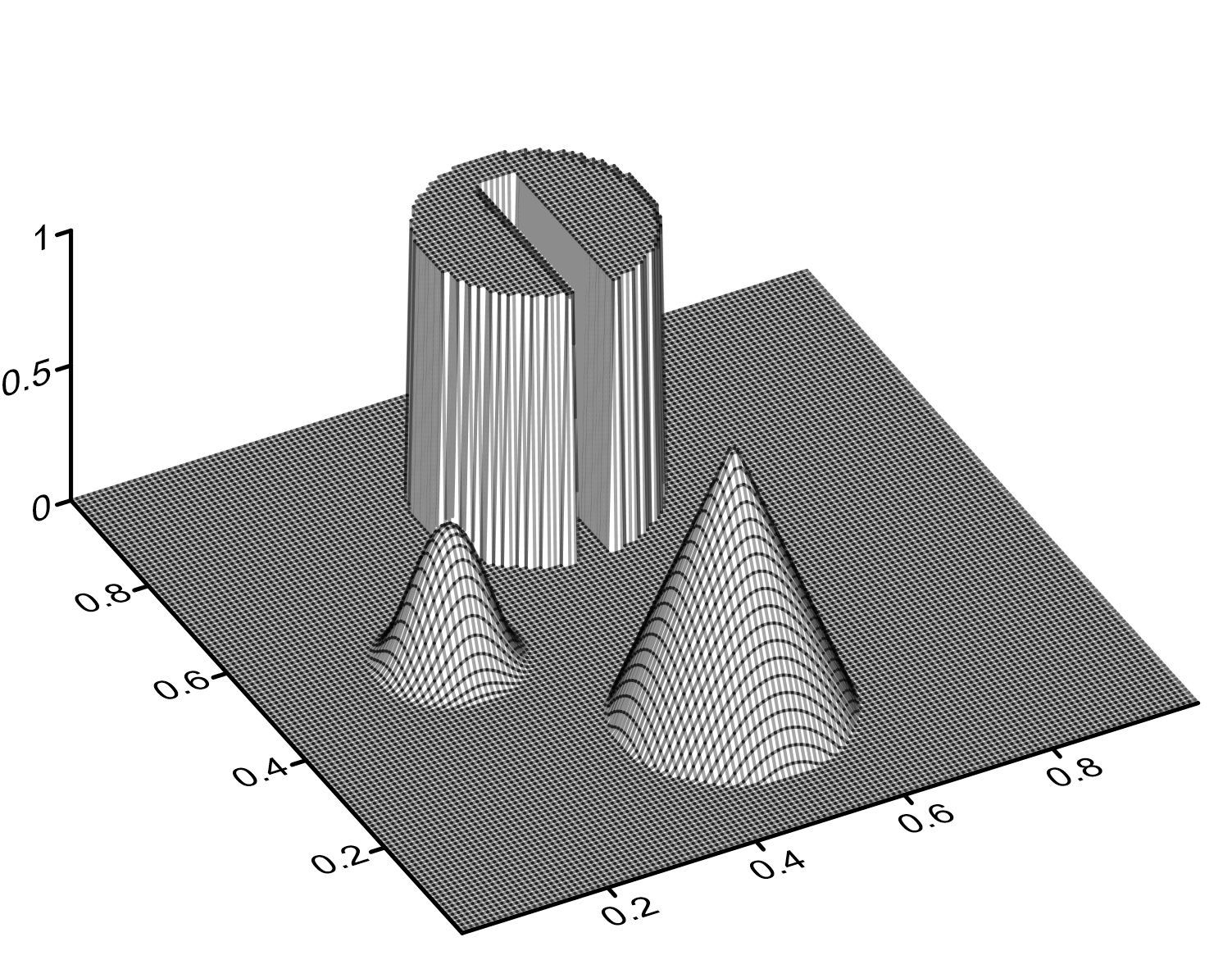}
\caption{Initial data and exact solution at the final time for the solid body rotation test}
\label{fig:4}       
\end{figure}


\subsection{Solid Body Rotations}  \label{Sec5_2}

In this section, we consider the rotation of solid bodies \cite{b33,b7,b2} under an incompressible flow that is described by the linear equation
\begin{equation}
\label{eq:58} 
 \frac{{\partial \rho }}{{\partial t}} + \boldsymbol u \cdot \nabla    {\rho }  = 0 \qquad 	\text{in} \quad \Omega  = \left( {0,1} \right) \times \left( {0,1} \right)
\end{equation} 							
with zero boundary conditions. The initial condition includes a slotted cylinder, a cone and a smooth hump (Fig.~\ref{fig:4}).  The slotted cylinder of radius 0.15 and height 1 is centered at the point (0.5,0.75) and
\[
  \rho (x,y,0) =\begin{cases}
  1 \qquad  {\rm if} \;\; \left| {x - 0.5} \right| \ge 0.025 \;\; {\rm or} \;\; y \ge 0.85\\
  0 \qquad  \rm{otherwise}
\end{cases}   
\]

\begin{figure*}[!t]
\centering
  \includegraphics[width=0.3\textwidth]{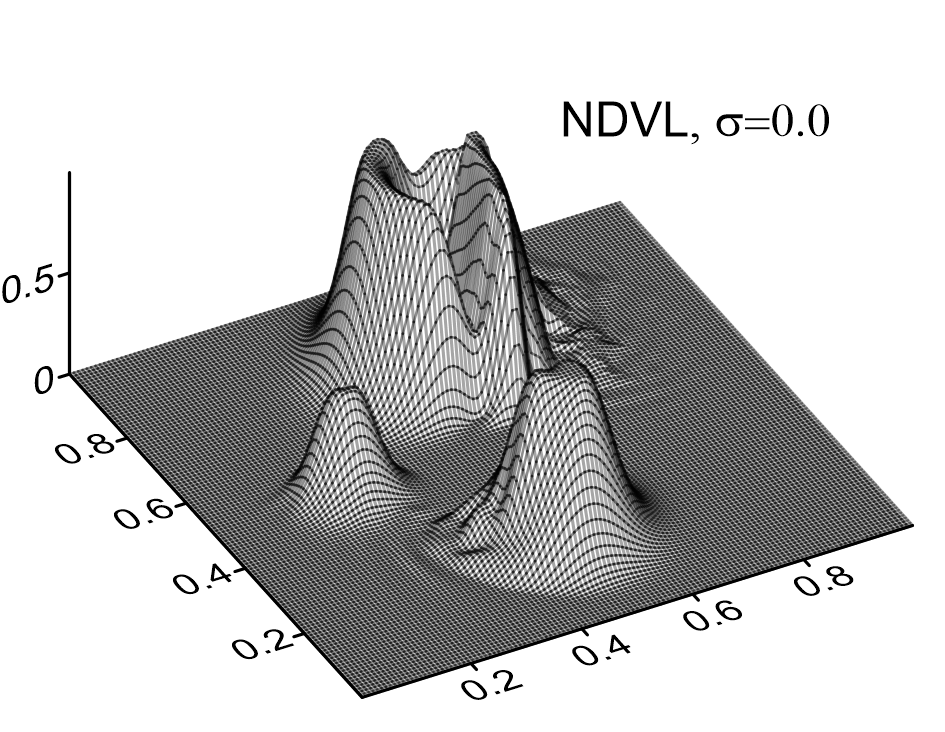}
  \includegraphics[width=0.3\textwidth]{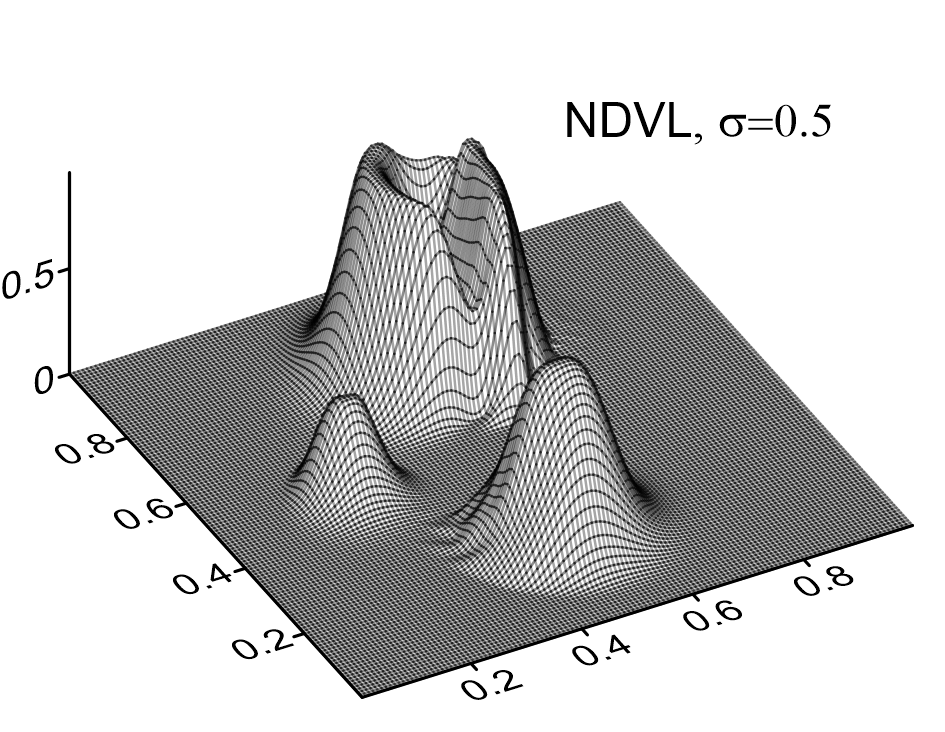}    \includegraphics[width=0.3\textwidth]{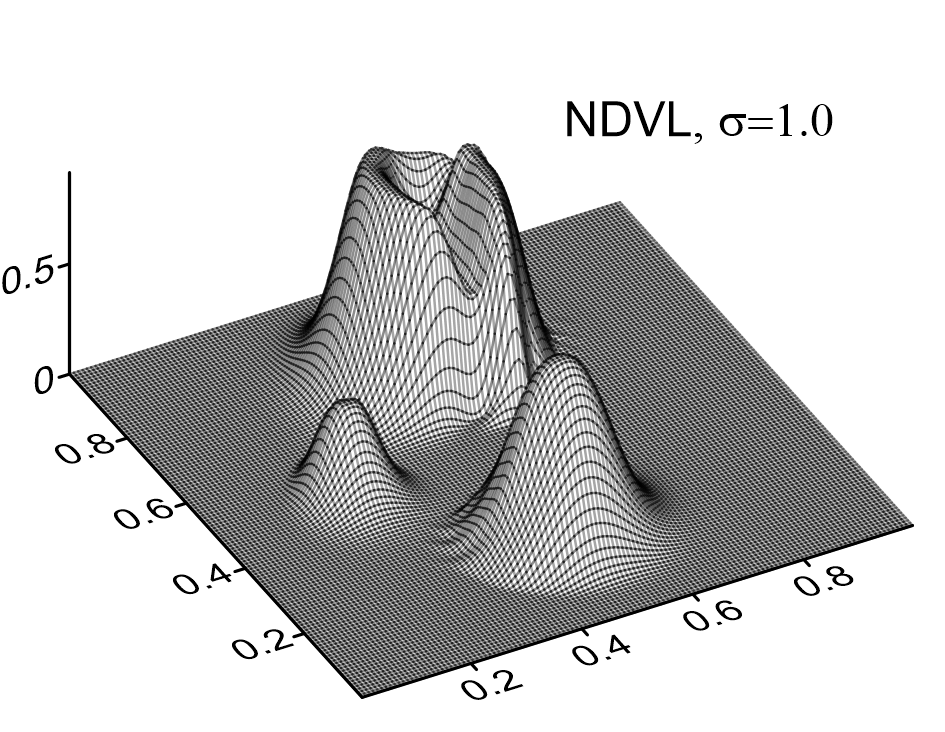}
  \includegraphics[width=0.3\textwidth]{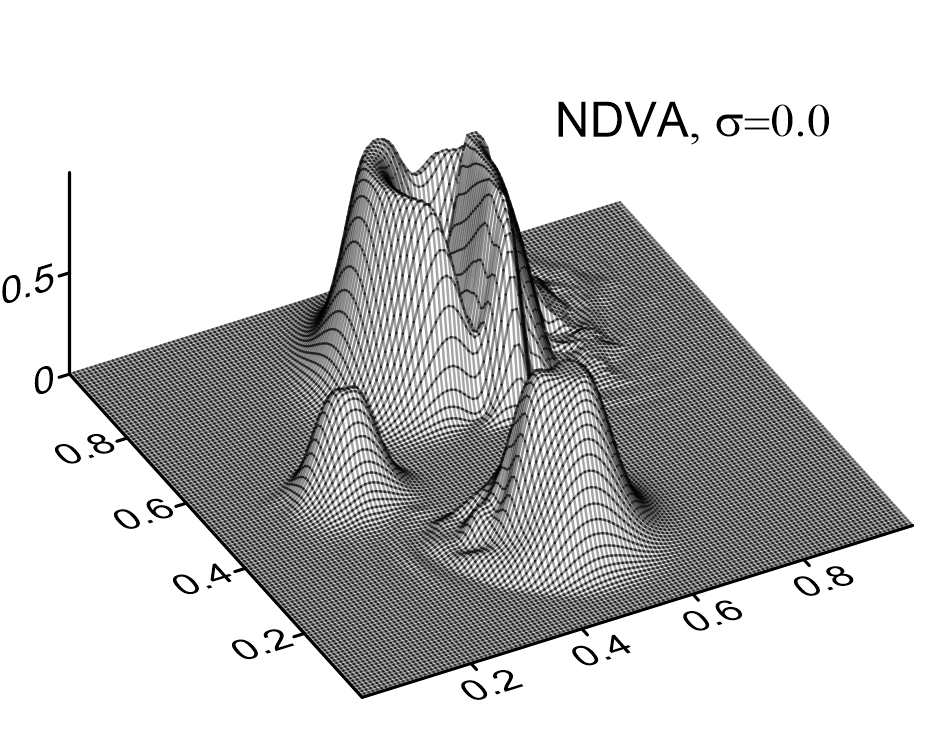}
  \includegraphics[width=0.3\textwidth]{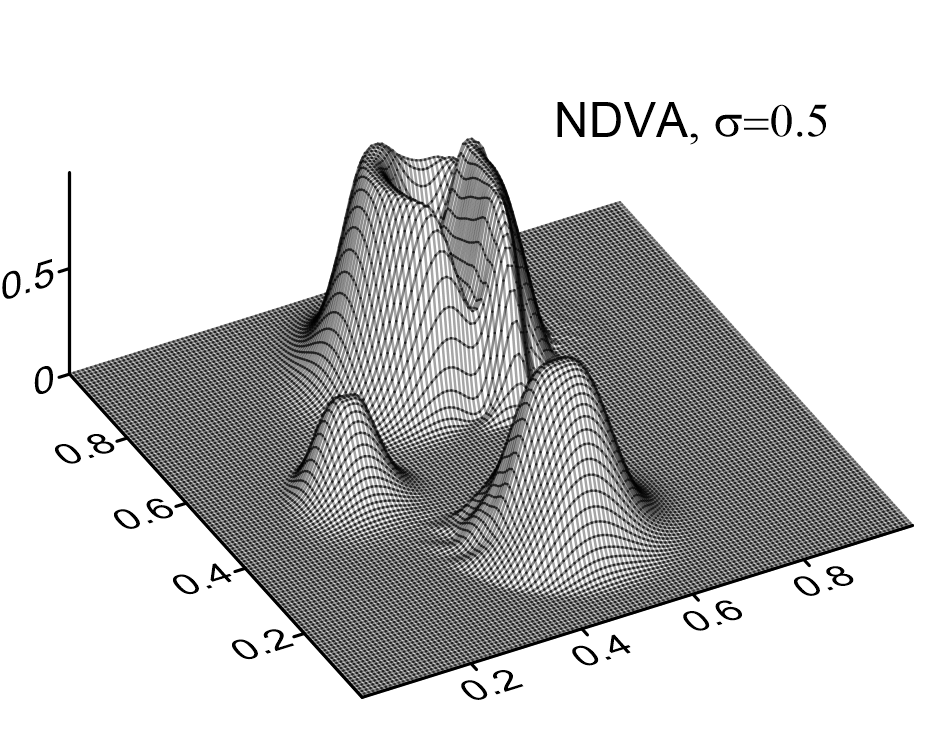}
  \includegraphics[width=0.3\textwidth]{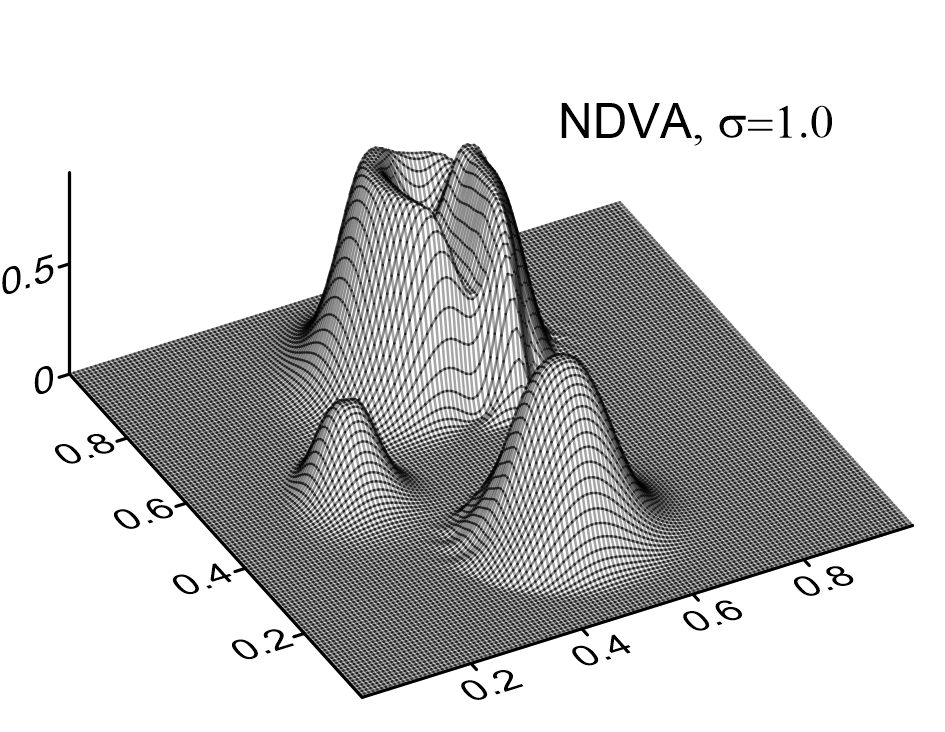}
  \includegraphics[width=0.3\textwidth]{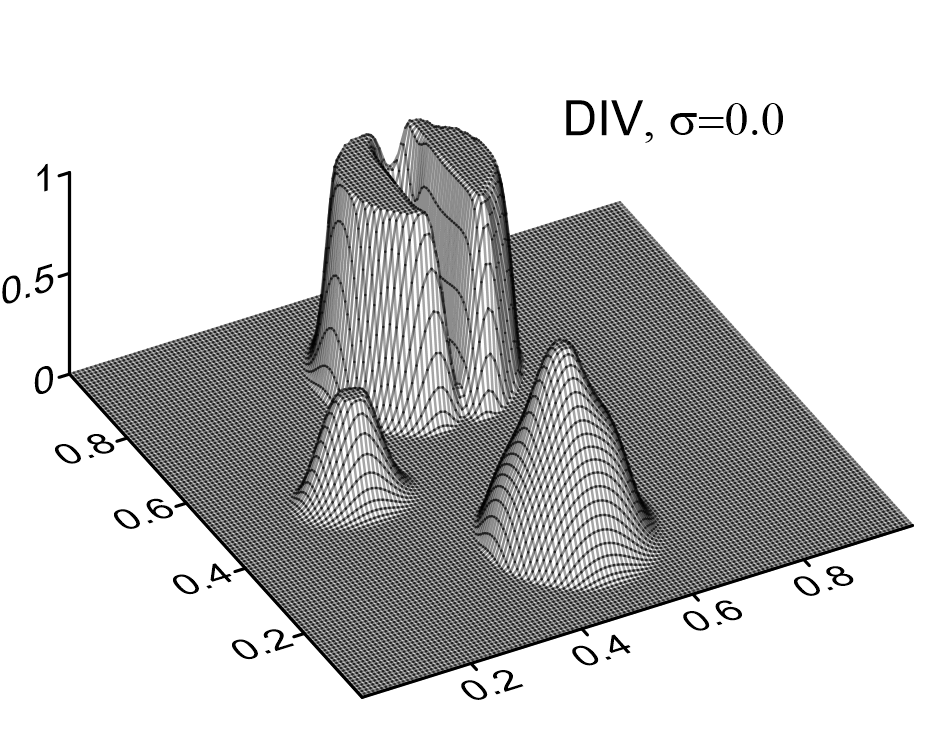}
  \includegraphics[width=0.3\textwidth]{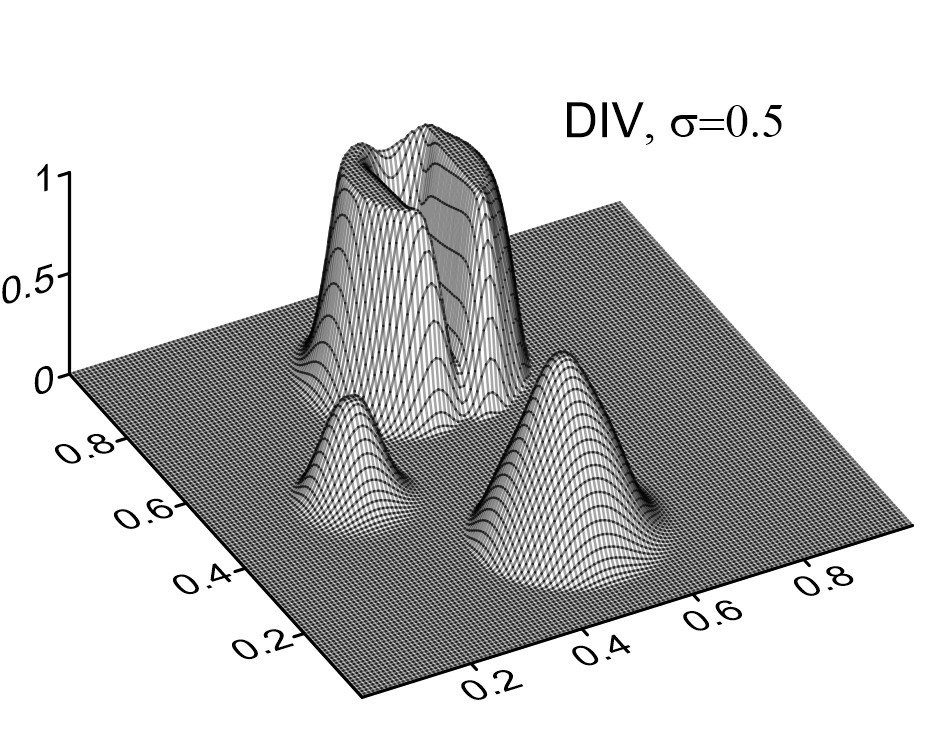}
  \includegraphics[width=0.3\textwidth]{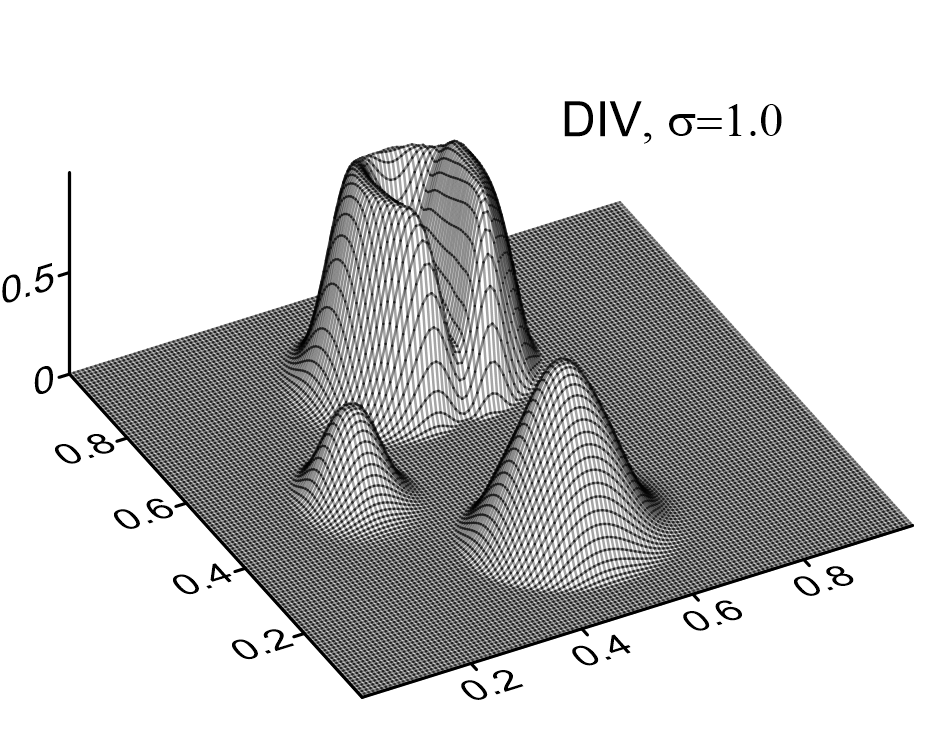}
\caption{Numerical results of the solid body rotation test after one revolution (5000 time steps) with the NDVL, NDVA, and DIV schemes  for various $\sigma$. }
\label{fig:5}       
\end{figure*}

\begin{table} [!b]
\caption{$L^1$-norm of errors and the maximum values of the numerical solutions for the solid body rotation test with the DIV, NDVL, and NDVA schemes}
\label{tab2}       
\begin{tabular}{llllllll}
\hline\noalign{\smallskip}
&  &  \multicolumn{2}{c}{DIV} & \multicolumn{2}{c}{NDVL} & \multicolumn{2}{c}{NDVA}  \\
 \cline{3-4} \cline{5-6} \cline{7-8} \noalign{\smallskip}
 & $\sigma$  &  $L^1$ error & $y_{max}$ & $L^1$ error & $y_{max}$  & $L^1$ error & $y_{max}$   \\
\noalign{\smallskip}\hline\noalign{\smallskip}
  
	& 0.0 &	2.5900$\times 10^{-2}$ & 1.0000 & 4.4189$\times 10^{-2}$ & 0.9959 & 4.4337$\times 10^{-2}$ & 0.9946 \\
{Cyl} & 0.5 &	2.8022$\times 10^{-2}$ & 0.9912 & 4.0252$\times 10^{-2}$ & 0.9548 & 4.0256$\times 10^{-2}$ & 0.9547 \\
	& 1.0 &	3.0557$\times 10^{-2}$ & 0.9681 & 3.9751$\times 10^{-2}$ & 0.9141  & 3.9749$\times 10^{-2}$ & 0.9139 \\
\hline \noalign{\smallskip}
	& 0.0 &	2.9773$\times 10^{-3}$ & 0.8709 & 3.4419$\times 10^{-3}$ & 0.8144 & 3.4419$\times 10^{-3}$ & 0.8143 \\
{Cn}	& 0.5 &	2.1664$\times 10^{-3}$ & 0.8434 & 2.6798$\times 10^{-3}$ & 0.8094 & 2.6799$\times 10^{-3}$ & 0.8092 \\
	& 1.0 &	2.4633$\times 10^{-3}$ & 0.8190 & 2.8654$\times 10^{-3}$ & 0.7905 & 2.8655$\times 10^{-3}$ & 0.7905 \\
\hline \noalign{\smallskip}	
	& 0.0 &	1.2495$\times 10^{-3}$ & 0.4947 & 2.1282$\times 10^{-3}$ & 0.4808 & 2.1283$\times 10^{-3}$ & 0.4804 \\
{Hm}	& 0.5 &	1.2132$\times 10^{-3}$ & 0.4645 & 1.7634$\times 10^{-3}$ & 0.4248 & 1.7636$\times 10^{-3}$ & 0.4247 \\
	& 1.0 &	1.4077$\times 10^{-3}$ & 0.4247 & 1.7701$\times 10^{-3}$ & 0.3869 & 1.7703$\times 10^{-3}$ & 0.3868\\
\noalign{\smallskip}\hline	\noalign{\smallskip}
\end{tabular} 
Cyl = Slotted Cylinder; Cn = Cone; Hm = Hump. 
\end{table}

The cone of also radius $r_0=0.15$ and height 1 is centered at point $({x_0},{y_0}) = (0.25,0.5)$ and
\[ \rho (x,y,0) = 1 - r(x,y) \] 
where 
\[r(x,y) = \frac{{\min (\sqrt {{{(x - {x_0})}^2} + {{(y - {y_0})}^2}} ,{r_0})}}{{{r_0}}} \]

The hump is given by
 \[ \rho (x,y,0) = \frac{1}{4}(1 + \cos (\pi r(x,y)) \]
where $({x_0},{y_0}) = (0.5,0.25)$ and $r_0=0.1$.

The flow velocity is calculated by $\boldsymbol u(x,y) = \left( { - 2\pi (y - 0.5),2\pi (x - 0.5)} \right)$ and in result of which the counterclockwise rotation takes place about domain point (0.5, 0.5). The computational grid consists of uniform $128 \times 128$ cells. The exact solution of \eqref{eq:58} reproduces by the initial state after each full revolution.

The numerical results produced with the NDVL, NDVA, and DIV schemes  after one full revolution (5000 time steps) with different weights $ \sigma $ are presented in Fig.~\ref{fig:5}. The $L^1$-norm of errors and the maximum values of the numerical results are given in Table~\ref{tab2}.
As in the above advection test, we also note a good agreement between the numerical results obtained with the NDVL and NDVA schemes. Again, the solution obtained by the DIV scheme is more accurate than the solutions computed by the NDVL and NDVA schemes.



\section{Conclusions}  \label{Sec6}
In this paper, we derive the formulas for calculating flux limiters for the FCT method for a nonconservative convection-diffusion equation. The flux limiter is computed as an approximate solution of the  optimization problem that can be considered as a background of the FCT approach.

Following FCT, we consider  a hybrid scheme which is a linear combination of monotone and high-order schemes. The difference between high-order flux and low-order flux is considered as an antidiffusive flux. The finding maximal flux limiters for the  antidiffusive fluxes is treated as an optimization problem with a linear objective function. Constraints for the optimization problem are inequalities that are valid for the monotone scheme and applied to the hybrid scheme. This approach allows us to reduce classical two-step FCT to a one-step method for explicit difference schemes and design flux limiters with desired properties.

Numerical experiments show the best results are obtained for the flux correction for the divergent part of the convective flux of a nonconservative convection-diffusion equation. We also note a good agreement between the numerical results for which the flux limiters are computed using exact and approximate solutions of optimization problem.


%
%



\end{document}